\documentclass[10pt,a4paper]{article}
\usepackage{amsfonts}
\usepackage[latin9]{inputenc}
\usepackage{amsmath}
\usepackage{amssymb}
\usepackage{breqn}
\usepackage{url}
\usepackage{graphicx}
\usepackage{bbm}
\usepackage[pdfstartview=FitH]{hyperref}
\usepackage{amsthm}
\usepackage{authblk}
\usepackage{hyperref}
\usepackage{url}
\usepackage{appendix}
\usepackage{minitoc}
\makeatletter
\begin{document}
\newtheorem{theorem}{Theorem}[section]
\newtheorem*{theorem*}{Theorem}
\newtheorem{lemma}[theorem]{Lemma}
\newtheorem{definition}[theorem]{Definition}
\newtheorem{claim}[theorem]{Claim}
\newtheorem{example}[theorem]{Example}
\newtheorem{remark}[theorem]{Remark}
\newtheorem{proposition}[theorem]{Proposition}
\newtheorem{corollary}[theorem]{Corollary}

\title{Isoperimetric Inequalities and topological overlapping for quotients of Affine buildings}
\author{Izhar Oppenheim}
\affil{Department of Mathematics\\
 The Ohio State University  \\
 Columbus, OH 43210, USA \\
E-mail: izharo@gmail.com}

\maketitle
\textbf{Abstract}. We prove isoperimetric inequalities for quotients of $n$-dimensional Affine buildings. We use these inequalities to prove topological overlapping for the $2$-dimensional skeletons of these buildings.   

\section{Introduction}
The notion of topological overlapping was defined by Gromov in \cite{Grom} as:

\begin{definition}
Let $X$ be an $n$-dimensional simplicial complex. Given a map $f : X^{(0)} \rightarrow \mathbb{R}^n$ (where $X^{(0)}$ are the vertices of $X$), a topological extension of $f$ is a continuous map $\widetilde{f} : X \rightarrow \mathbb{R}^n$ which coincides with $f$ on $X^{(0)}$.
A simplicial complex $X$ is said to have $c$-topological overlapping (with $1 \geq c>0$) if for every $f: X^{(0)} \rightarrow \mathbb{R}^n$ and every topological extension $\widetilde{f}$, there is a point $z \in \mathbb{R}^n$ such that 
$$\vert \lbrace \sigma \in X^{(n)} : z \in \widetilde{f} (\sigma)\rbrace \vert \geq c \vert X^{(n)} \vert ,$$
(where $X^{(n)}$ are the $n$-dimensional simplices of $X$). 
In other words, this means that at least a $c$ fraction of the images of $n$-simplices intersect at a single point.  \\
A family of pure $n$-dimensional simplicial complexes $\lbrace X_j \rbrace$ is called a family of topological expanders, if there is some $c>0$ such that for every $j$, $X_j$ has  $c$-topological overlapping. \\
\end{definition}
In \cite{Grom}, Gromov gave examples of families of topological expanders, but in all the examples the degree of the vertices was unbounded. This raised the question if there are families of topological expanders with a bounded degree on the vertices. This question received a positive answer in \cite{KKL}, where Kaufman, Kazhdan and Lubotzky showed that the $2$-skeleton of (non partite) $3$-dimensional Ramanujan complexes with large enough thickness, has topological overlapping depending only on the thickness. \\
However, some arguments given in \cite{KKL} relayed heavily on the unique structure of $3$-dimensional Ramanujan complexes. In this article we generalize the results of \cite{KKL}, such that they will hold for quotients of $n$-dimensional (for $n>2$) classical affine buildings of any type, assuming large enough thickness. We should remark that a lot of the ideas we use in our proofs already appear in some form in \cite{KKL}. \\
As in \cite{KKL}, we derive topological overlapping from higher isoperimetric inequalities. To state these inequalities, we shall need some additional definitions given below (this definitions are repeated and expended in subsection \ref{subsection F_2 expansion} below). \\
Let $X$ be a pure $n$-dimensional simplicial complex let $-1 \leq k \leq n$. Denote by $X^{(k)}$ the set of simplices of dimension $k$ and define a weight function 
$$m: \bigcup_{k=-1}^n X^{(k)} \rightarrow \mathbb{R} ,$$
$$\forall \sigma \in \bigcup_{k=-1}^n X^{(k)}, m (\sigma) =  \dfrac{(n-k)! \vert \lbrace \eta \in X^{(n)} : \sigma \subseteq \eta \rbrace \vert}{\vert X^{(n)} \vert}.$$
Denote  $C^{k} (X,\mathbb{F}_2)$ to be the $k$-cochains with coefficients in $\mathbb{F}_2$, i.e., $C^{k} (X,\mathbb{F}_2) = \lbrace \phi : X^{(k)} \rightarrow \mathbb{F}_2 \rbrace$ and define a norm on $C^k  (X,\mathbb{F}_2)$ as 
$$\Vert \phi \Vert = \sum_{\sigma \in X^{(k)}, \phi (\sigma)=1} m(\sigma).$$
Next, define the differential $d_k :  C^k  (X,\mathbb{F}_2) \rightarrow C^{k+1} (X,\mathbb{F}_2)$ as 
$$\forall \eta \in X^{(k+1)}, d_k \phi (\eta) = \sum_{\sigma \in X^{(k)}, \sigma \subset \eta} \phi (\sigma),$$
where the addition above is in $\mathbb{F}_2$. 
For $k \geq 0$, define 
$$B^k (X,\mathbb{F}_2) = \lbrace d_{k-1} \phi : \phi \in C^{k-1} (X,\mathbb{F}_2) \rbrace \subseteq C^k (X,\mathbb{F}_2).$$
A cochain $\phi \in C^k (X,\mathbb{F}_2)$ is called minimal if 
$$\Vert \phi \Vert \leq \Vert \phi - \psi \Vert, \forall \psi \in B^k (X,\mathbb{F}_2).$$
\begin{remark}
The reader should note that our definition of the norm is slightly different than the one given in \cite{KKL} (and in other references). This is done for technical reasons and any of our results can be easily translated to results in the norm of \cite{KKL}. This is explained in further detail in remark \ref{weights remark} below. 
\end{remark}
Next, we recall the definition of local spectral expansion given by the author in \cite{Opp-LocalSpectral} (a reader not familiar with the definition of links can find it in subsection \ref{Links and spectral gaps subsection} below):
\begin{definition}
A  pure $n$-dimensional simplicial complex is said to have $\lambda$-local spectral expansion, where $\frac{n-1}{n} < \lambda \leq 1$ if both of the following hold:
\begin{enumerate}
\item All the links of $X$ of dimension $>0$ are connected. 
\item All the $1$-dimensional links have spectral gap at least $\lambda$ (the $1$-dimensional links are the links of the $(n-2)$-simplices). 
\end{enumerate}
\end{definition}

Our main results giving isoperimetric inequalities for $n$-dimensional affine buildings with large local spectral expansion. For $1$-cochains an isoperimetric inequality can be deduced based on the local spectral gap alone:
\begin{theorem}
\label{isoperimetric inequalities theorem - introduction, k=1}
There is a constant $\theta <1$ such that for every pure $n$-dimensional simplicial complex with $\lambda$-local spectral gap with $\lambda \geq \theta$, we have that for every $\phi \in C^1 (X, \mathbb{F}_2)$, if $\Vert \phi \Vert \leq 12 C_1$ and $\phi$ is minimal, then $\Vert d \phi \Vert \geq \frac{1}{4} \Vert \phi \Vert$. 
\end{theorem}
However, in order to deduce topological overlapping, we need similar isoperimetric inequalities for $2$-cochains. Currently (when writing this article), we do not know how to prove such inequalities using spectral properties alone and therefore we'll have to add an assumption regarding the coboundary expansion of the links of vertices:
\begin{theorem}
For every $n >2$, $\epsilon >0$, there are constants $\frac{n-1}{n} < \Lambda_2 <1, C_2 >0$ such that for every pure $n$-simplicial complex $X$ of dimension $n >2$ with $\lambda$-local spectral expansion, if $\lambda \geq \Lambda_2$ and if for every $\lbrace v \rbrace \in X^{(0)}$
$$\min \left\lbrace \dfrac{\Vert d \phi \Vert}{\min_{\psi \in B^1 (X_{\lbrace v \rbrace}, \mathbb{F}_2)} \Vert \phi - \psi \Vert} : \phi \in C^{1} (X_{\lbrace v \rbrace}, \mathbb{F}_2) \setminus B^1 (X_{\lbrace v \rbrace}, \mathbb{F}_2) \right\rbrace \geq \epsilon,$$
(where  $X_{\lbrace v \rbrace}$ it the link of $\lbrace v \rbrace$ - see further explanation below), we have for every $\phi \in C^2 (X,\mathbb{F}_2)$ that 
$$\left( \phi \text{ is minimal and } \Vert \phi \Vert \leq 24 C_2 \right) \Rightarrow \Vert d \phi \Vert \geq \dfrac{3 \epsilon}{10} \Vert \phi \Vert.$$
\end{theorem}

Using a result stated from \cite{KKL} (see theorem \ref{overlapping from KKL} below), we use (a slightly more general version) of the above isoperimetric inequalities to deduce the following topological overlapping result:

\begin{theorem}
\label{topological overlapping - introduction}
Let $\widetilde{X}$ be an $n$ dimensional ($n>2$) affine building that arises from group $G$ with an affine  $BN$-pair constructed over a non-archimedean local field $F$. Denote the thickness of $\widetilde{X}$ by $t$.  Let $\Gamma$ be a subgroup acting on $\widetilde{X}$ simplicially and cocompactly, such that for every vertex $v$ of $\widetilde{X}$ we have that
$$\forall g \in \Gamma, d(g.v, v)>2.$$
Then for $t$ large enough, there is a constant $c >0$ that depends only on $t$ and the type of $\widetilde{X}$ (i.e., on the Weyl group of $\widetilde{X}$), but not on $\Gamma$, such that the $2$-skeleton of $\widetilde{X} / \Gamma$ has $c$-topological overlapping property. 
\end{theorem}

\begin{corollary}
Let $\widetilde{X}$ as in the above theorem with large thickness. Then for a sequence of groups $\Gamma_i$ that act on $\widetilde{X}$ simplicially and cocompactly with the condition stated in the above theorem. Let $Y_i = \widetilde{X} / \Gamma_i$ and let $X_i$ be the $2$-skeleton of $Y_i$. Then $\lbrace X_i \rbrace$ is a family of topological expanders.
\end{corollary}

\begin{remark}
Explicit examples of affine buildings $\widetilde{X}$ and $\Gamma_i$'s as in the corollary are the Ramanujan complexes constructed in \cite{LSV}.
\end{remark}

\begin{remark}
As noted above, it is our hope to prove the isoperimetric inequalities stated above based on spectral information only, i.e., for a general simplicial complex with large enough local spectral expansion without assuming the complex to be an affine building. The question how to do so (or if this can be done at all) is left for future study.
\end{remark}

\textbf{Structure of this article:} In section 2, we review basic definitions and results about weighted graphs, weighted complex and links. In section 3, we review different notions of high dimensional expansion considered in the article. In section 4, we prove some new technical results connecting the norm of $\mathbb{F}_2$ cochains to the norms of the localizations. In section 5, we discuss different notions of minimality for $\mathbb{F}_2$ cochains. In section 6, we discuss criteria to topological expansion and show how to drive topological expansion from (certain types of) isoperimetric inequalities.  In section 7, we prove the isoperimetric inequalities stated above (which are the main results of this article). Finally, in section 8, we derive topological overlapping for the $2$-skeleton of affine building. In the interest of readers not used to the weighted setting, we added an appendix covering basic results (such as the Cheeger inequality) for weighted graphs.

\section{Background definitions and results}
This section is aimed to provide background results that we shall need in order to prove our main theorems. 
\subsection{Weighted graphs and Cheeger inequalities}
We shall provide the basic definitions on weighted graphs, graph Laplacians and state (without proof) some Cheeger inequalities for weighted graphs. The interested reader can find proofs and a more complete discussion regarding the definitions in the appendix. \\
For a graph $G=(V,E)$, a weight function is a function $m: V \cup E \rightarrow \mathbb{R}^+$ such that 
$$\forall v \in V, m (v) = \sum_{e \in E, v \in e} m(e) .$$ 
Denote by $C^0 (G, \mathbb{R})$ the space:
$$C^0 (G, \mathbb{R} ) = \lbrace \phi : V \rightarrow \mathbb{R} \rbrace$$
Let $\Delta^+ : C^0 (G, \mathbb{R} ) \rightarrow C^0 (G, \mathbb{R} )$ be the graph Laplacian with respect to the weight function $m$ which is defined as follows:
$$\Delta^+ \phi (v) =  \phi (v) - \dfrac{1}{m(v)} \sum_{u \in V, \lbrace u,v \rbrace \in E} m (\lbrace u,v \rbrace) \phi (u) .$$
The Laplacian is a positive operator and if $G$ is connected, only constant functions have the eigenvalue $0$. For a connected graph $G$, we denote by $\lambda (G)$ the smallest positive eigenvalue. Next, we state the following Cheeger-type inequalities (the reader can find the proofs in the appendix): 
\begin{proposition}
\label{Cheeger for weighted graphs proposition}
Let $G = (V,E)$ be a connected graph. We introduce the following notations:
For $\emptyset \neq U \subseteq V$ denote 
$$m (U) = \sum_{v \in U} m (v) .$$
For $\emptyset \neq U_1, U_2 \subseteq V$ denote
$$m(U_1,U_2 ) = \sum_{u_1 \in U_1, u_2 \in U_2,  \lbrace u_1, u_2 \rbrace \in E} m (\lbrace u_1, u_2 \rbrace ).$$
Then:
\begin{enumerate}
\item (Cheeger inequality) For every $\emptyset \neq U \subsetneq V$, we have that
$$m (U, V \setminus U) \geq \lambda (G) \dfrac{m (U) m (V \setminus U)}{m(V)} .$$
\item For every $\emptyset \neq U \subsetneq V$, we have that
$$\dfrac{m(U)}{2}  \left(1 - \lambda (G) \frac{m(V \setminus U)}{m(V)} \right) \geq m (U,U)  .$$
\end{enumerate}
\end{proposition}

\subsection{Weighted simplicial complexes}
Let $X$ be a pure $n$-dimensional finite simplicial complex. For $-1 \leq k \leq n$, we denote $X^{(k)}$ to be the set of all $k$-simplices in $X$ ($X^{(-1)} = \lbrace \emptyset \rbrace$). A weight function $m$ on $X$ is a function:
$$m : \bigcup_{-1 \leq k \leq n} X^{(k)} \rightarrow \mathbb{R}^+,$$
such that for every $-1 \leq k \leq n-1$ and for every $\tau \in X^{(k)}$ we have that
$$m (\tau) = \sum_{\sigma \in X^{(k+1)}} m(\sigma).$$
By its definition, it is clear the $m$ is determined by the values in takes in $X^{(n)}$. A simplicial complex with a weight function will be called a weighted simplcial complex. \\

\begin{proposition}
\label{weight in n dim simplices}
For every $-1 \leq k \leq n$ and every $\tau \in X^{(k)}$ we have that
$$m (\tau) =(n-k)!  \sum_{\sigma \in X^{(n)}, \tau \subseteq \sigma} m (\sigma ),$$
where $\tau \subseteq \sigma$ means that $\tau$ is a face of $\sigma$. 
\end{proposition}

\begin{proof}
The proof is by induction. For $k=n$ this is obvious. Assume the equality is true for $k+1$, then for $\tau \in X^{(k)}$ we have
\begin{dmath*}
m( \tau ) = \sum_{\sigma \in X^{(k+1)}, \tau \subset \sigma} m(\sigma ) = \sum_{\sigma \in X^{(k+1)}, \tau \subset \sigma} (n-k-1)! \sum_{\eta \in X^{(n)}, \sigma \subset \eta} m(\eta) = (n-k) (n-k-1)! \sum_{\eta \in X^{(n)}, \tau \subset \eta} m(\eta)=   (n-k)! \sum_{\eta \in X^{(n)}, \tau \subset \eta} m(\eta) .
\end{dmath*}
\end{proof}

\begin{corollary}
\label{weight in l dim simplices}
For every $-1 \leq k < l \leq n$ and every $\tau \in X^{(k)}$ we have
$$ m(\tau) =(l-k)! \sum_{\sigma \in X^{(l)}, \tau \subset \sigma} m(\sigma) .$$
\end{corollary}

\begin{proof}
For every $\sigma \in X^{(l)}$ we have
$$ m(\sigma ) = (n-l)! \sum_{\eta \in X^{(n)}, \sigma \subseteq \eta} m(\eta) .$$
Therefore
\begin{dmath*}
\sum_{\sigma \in X^{(l)}, \tau \subset \sigma} m(\sigma) = \sum_{\sigma \in X^{(l)}, \tau \subset \sigma} (n-l)! \sum_{\eta \in X^{(n)}, \sigma \subseteq \eta} m(\eta) = \dfrac{(n-k)!}{(l-k)! (n-k - (l-k) )! } (n-l)! \sum_{\eta \in X^{(n)}, \tau \subseteq \eta} m(\eta) = \dfrac{(n-k)!}{(l-k)!}  \sum_{\eta \in X^{(n)}, \tau \subseteq \eta} m(\eta) = \dfrac{1}{(l-k)!} m (\tau ) .
\end{dmath*}
\end{proof}

For $-1 \leq k \leq n$ and a set $\emptyset \neq U \subseteq X^{(k)}$, we denote 
$$m(U) = \sum_{\sigma \in U} m(\sigma ).$$

\begin{proposition}
\label{weight of X^k compared to X^l}
For every $-1 \leq k < l \leq n$, 
$$m(X^{(k)}) = \dfrac{(l+1)!}{(k+1)!} m (X^{(l)}). $$
\end{proposition}

\begin{proof}
By corollary \ref{weight in l dim simplices}, we have that
\begin{dmath*}
m(X^{(k)}) = \sum_{\tau \in X^{(k)}} m(\tau ) = \sum_{\tau \in X^{(k)}} (l-k)! \sum_{\sigma \in X^{(l)}, \tau \subset \sigma} m(\sigma) =\sum_{\sigma \in X^{(l)}} (l-k)! m(\sigma) \sum_{\tau \in X^{(k)}, \tau \subset \sigma} 1 = \sum_{\sigma \in X^{(l)}} (l-k)! m(\sigma) {l+1 \choose k+1} = \dfrac{(l+1)!}{(k+1)!} m(X^{(l)}) .
\end{dmath*}
\end{proof}

\begin{remark}
\label{1 skeleton of weighted complex is weighted graph - remark} 
Note that for every weighted complex $X$, the $1$-skeleton of $X$ is a weighted graph and all the results stated above for weighted graphs hold.
\end{remark} 

We would like to distinguish to following weight function $m_h$ which we call the homogeneous weight function (since it give the value $1$ to each $n$-dimensional simplex):
$$\forall \tau \in X^{(k)},   m_h (\tau) =  (n-k)! \vert \lbrace \eta \in X^{(n)} : \tau \subseteq \eta \rbrace \vert .$$
The next proposition shows that $m_h$ is indeed a weight function:
\begin{proposition}
For every $-1 \leq k \leq n-1$ and every $\tau \in X^{(k)}$ we have that 
$$m_h (\tau ) = \sum_{\sigma \in X^{(k+1)}, \tau \subset \sigma} m_h (\sigma).$$
\end{proposition}
\begin{proof}
Fix  $-1 \leq k \leq n-1$ and $\tau \in X^{(k)}$, note that for every $\eta \in X^{(n)}$ with $\tau \subset \eta$, there are exactly $n-k$ simplices $\sigma \in X^{(k+1)}$ such that $\tau \subset \sigma \subseteq \eta$. Therefore we have that
\begin{dmath*}
\sum_{\sigma \in X^{(k+1)}, \tau \subset \sigma} m_h (\sigma) = {(n-k-1)! \sum_{\sigma \in X^{(k+1)}, \tau \subset \sigma} \vert \lbrace \eta \in X^{(n)} : \sigma \subseteq \eta \rbrace \vert} = {(n-k)!  \vert \lbrace \eta \in X^{(n)} : \tau \subseteq \eta \rbrace \vert} = m_h (\tau).
\end{dmath*}
\end{proof}

It will be more convenient to normalize $m_h$ as follows: define $\overline{m_h}$ to be the weight function:
$$\forall \sigma \in X^{(k)}, \overline{m_h} = \dfrac{m_h (\sigma )}{\vert X^{(n)} \vert} .$$

\begin{remark}
We remark that most of our results holds for any weight function. The only place where we'll need to use the normalized homogeneous weight function is when we like to deduce topological overlap using the results in \cite{KKL}.
\end{remark}

Throughout this article, $X$ is a pure, $n$-dimensional weighted simplicial complex with a weight function $m$.

\subsection{Links and spectral gaps}
\label{Links and spectral gaps subsection}
Let $X$ be a pure $n$-dimensional finite simplicial complex. For $\lbrace v_{0},...,v_{j} \rbrace=\tau\in X^{(j)}$, denote by $X_{\tau}$
the \emph{link} of $\tau$ in $X$, that is, the (pure) complex of dimension
$n-j-1$ consisting on simplices $\sigma=\lbrace  w_{0},...,w_{k} \rbrace$
such that $\lbrace v_{0},...,v_{j} \rbrace, \lbrace w_{0},...,w_{k} \rbrace$ are disjoint as sets and $\lbrace v_{0},...,v_{j} \rbrace \cup \lbrace w_{0},...,w_{k} \rbrace \in X^{(j+k+1)}$. Note that for $\lbrace \emptyset \rbrace = X^{(-1)}$, $X_\emptyset =X$. \\
If $m$ is a weight function on $X$, define $m_\tau$ to be a weight function of $X_\tau$ by taking
$$m_\tau (\sigma ) = m( \tau \cup \sigma).$$
\begin{proposition}
For $0 \leq j \leq n-2$ and $\tau \in X^{(j)}$, the function $m_\tau$ defined above is indeed a weight function on $X_\tau$, i.e., for every $k<n-j-1$ and every $\sigma \in X_\tau^{(k)}$ we have that 
$$m_\tau (\sigma ) = \sum_{\eta \in X_\tau^{(k+1)}} m_\tau (\eta ) .$$
\end{proposition}

\begin{proof}
Let $k<n-j-1$ and $\sigma \in X_\tau^{(k)}$, then
\begin{dmath*}
m_\tau (\sigma ) = m (\tau \cup \sigma ) = \sum_{\gamma \in X^{(k+j+1)}, \tau \cup \sigma \subset \gamma} m(\gamma) = \sum_{\eta \in X^{(k)}, \tau \cup \eta \in X^{(k+j+1)}, \sigma \subset \eta} m(\tau \cup \eta) =  \sum_{\eta \in X_\tau^{(k+1)}} m_\tau (\eta ) .
\end{dmath*}
\end{proof}

By remark \ref{1 skeleton of weighted complex is weighted graph - remark}, the $1$-skeleton of $X_\tau$ is a weighted graph with the weight function $m_\tau$. We shall denote by $\lambda (X_\tau)$ the smallest positive eigenvalue of the (weighted) graph Laplacian on the graph. 
\begin{remark}
We remark that when one works with the homogeneous weight function $m_h$, then for $\tau \in X^{(n-2)}$, $X_\tau$ is a graph and $m_{h,\tau}$ assigns the weight $1$ to each edge, so in this setting, the Laplacian is the usual (un-weighted) graph Laplacian.
\end{remark} 

Next, we'll state a useful theorem from \cite{Opp-LocalSpectral}:
\begin{theorem}\cite{Opp-LocalSpectral}[Lemma 5.1, Corollary 5.2]
\label{decent in links theorem}
Let $X$ be a pure $n$-dimensional weighted simplicial complex (with $n>1$). Assume that all the links of $X$ of dimension $>0$ are connected (including $X$ itself). For $-1 \leq k \leq n-2$ denote  
$$\lambda_k = \min_{\tau \in X^{(k-1)}} \lambda (X_\tau).$$
\begin{enumerate}
\item For every $0 \leq k \leq n-2$, we have that
$$\lambda_k \geq 2 - \dfrac{1}{\lambda_{k+1}} .$$
\item If $\lambda_{n-1} > \frac{n-1}{n}$, then for every $0 \leq k \leq n-2$, we have that $\lambda_{k} > \frac{k}{k+1}$. Moreover, for every $N \in \mathbb{N}, N \geq n$, we have that if 
 $\lambda_{n-1} > \frac{N-1}{N}$, then for every $0 \leq k \leq n-2$, $\lambda_{k} > \frac{N-n+k}{N-n + k+1}$.
\end{enumerate}

\end{theorem}

\section{Different notions of expansion}
Here we'll review different notions of expansion considered in this article.
\subsection{$\mathbb{F}_2$-expansion}
\label{subsection F_2 expansion}
For $-1 \leq k \leq n$, define $C^{k} (X, \mathbb{F}_2)$ to be $k$-cochains of $X$ over $\mathbb{F}_2$:
$$C^{k} (X, \mathbb{F}_2) = \lbrace \phi : X^{(k)} \rightarrow \mathbb{F}_2 \rbrace. $$
For $\phi \in C^{k} (X, \mathbb{F}_2)$ denote by $supp(\phi)$ the support of $\phi$:
$$supp (\phi ) = \lbrace \sigma \in X^{(k)} : \phi (\sigma) =1  \rbrace .$$  
Further define the differential $d_k : C^{k} (X, \mathbb{F}_2) \rightarrow C^{k+1} (X, \mathbb{F}_2)$ as:
$$\forall \sigma \in X^{(k+1)}, \forall \phi \in C^{k} (X, \mathbb{F}_2), d_k \phi (\sigma) = \sum_{\tau \in X^{(k)}, \tau \subset \sigma} \phi (\tau)  .$$
It is easy to check that $d_{k+1} d_k = 0$, therefore the usual cohomological definitions hold: i.e., denote
$$B^k (X, \mathbb{F}_2) = im (d_{k-1} ), Z^k (X, \mathbb{F}_2) = ker (d_k), H^k (X, \mathbb{F}_2) = \dfrac{Z^k (X, \mathbb{F}_2)}{B^k (X, \mathbb{F}_2)} .$$
From now on, when there is no chance for confusion, we shall omit the index on the differential and just denote it by $d$. \\
Define the following norm on $ C^{k} (X, \mathbb{F}_2)$: for every $\phi \in  C^{k} (X, \mathbb{F}_2)$ define
$$\Vert \phi \Vert  = \sum_{\tau \in X^{(k)}, \phi (\tau ) =1} m(\tau).$$
\begin{claim}
For every $k$ as above and for every $\phi \in  C^k (X, \mathbb{F}_2)$, we have that 
$$\Vert d \phi \Vert \leq \Vert \phi \Vert.$$
\end{claim}
\begin{proof}
Fix $\phi \in  C^k (X, \mathbb{F}_2)$. Note that for every $\sigma \in X^{(k+1)}$
$$\left( \forall \tau \in supp (\phi), \tau \not\subset \sigma \right) \Rightarrow d \phi (\sigma) =0.$$
Also note that for every $\sigma \in X^{(k+1)}$, $d \phi (\sigma) \leq 1$. Therefore
\begin{dmath*}
\Vert d \phi \Vert = \sum_{\sigma \in X^{(k+1)}} m(\sigma) d \phi (\sigma) \leq \sum_{\sigma \in X^{(k+1)}} m(\sigma) (\sum_{\tau \in X^{(k)}, \tau \in supp (\phi), \tau \subset \sigma} 1 )  = \\ \sum_{\tau \in X^{(k)}, \tau \in supp (\phi)} \sum_{\sigma \in X^{(k+1)}, \tau \subset \sigma} m(\sigma) = \sum_{\tau \in X^{(k)}, \tau \in supp (\phi)} m(\tau) = \Vert \phi \Vert. 
\end{dmath*}
\end{proof}
Using the above norm, we'll define the following constants: 
\begin{definition}
for $-1 \leq k \leq n-1$, and $X$ a simplicial complex, define
\begin{enumerate}
\item The k-th coboundary expansion of $X$:
$$\epsilon_k (X) = \min \left\lbrace \dfrac{\Vert d \phi \Vert}{\min_{\psi \in B^k (X, \mathbb{F}_2)} \Vert \phi - \psi \Vert} : \phi \in C^{k} (X, \mathbb{F}_2) \setminus B^k (X, \mathbb{F}_2) \right\rbrace.$$
\item The k-th cocycle expansion of $X$:
$$ \widetilde{\epsilon}_k (X) = \min \left\lbrace \dfrac{\Vert d \phi \Vert}{\min_{\psi \in Z^k (X, \mathbb{F}_2)} \Vert \phi - \psi \Vert} : \phi \in C^{k} (X, \mathbb{F}_2) \setminus Z^k (X, \mathbb{F}_2) \right\rbrace.$$
\item The k-th cofilling constant of $X$:
$$\mu_k (X) = \max_{0 \neq \phi \in B^{k+1} (X,\mathbb{F}_2 )} \left( \dfrac{1}{\Vert \phi \Vert} \min_{\psi \in   C^{k} (X,\mathbb{F}_2 ), d \psi = \phi} \Vert \psi \Vert \right) .$$
\end{enumerate}
\end{definition}
\begin{remark}
\label{connection between cofilling and cocycle expansion}
Recall that $ B^k (X, \mathbb{F}_2) \subseteq  Z^k (X, \mathbb{F}_2)$ and therefore $\widetilde{\epsilon}_k (X) \geq \epsilon_k (X)$. Also note that 
$$\mu_k (X)  = \dfrac{1}{\widetilde{\epsilon}_k (X)},$$
(the proof of this equality is basically unfolding the definitions of both constants and therefore it is left as an exercise to the reader).   
\end{remark}
\begin{remark}
\label{the case k=-1 always has epsilon =1}
Note that when $k=-1$, we have that $B^{-1} (X, \mathbb{F}_2) = Z^{-1} (X, \mathbb{F}_2) = \lbrace 0 \rbrace$. Also note that the only $\phi in  C^{-1} (X, \mathbb{F}_2) \setminus B^{-1} (X, \mathbb{F}_2)$ is $\phi (\emptyset) =1$  and for that $\phi$, we have that 
$$\Vert d \phi \Vert = m(X^{(0)}) = m(\emptyset) = \Vert \phi \Vert.$$
Therefore, we always have 
$$\mu_{-1} (X) = \widetilde{\epsilon}_{-1} (X) = \epsilon_{-1} (X) =1.$$ 
\end{remark}
\begin{definition}
Let $\lbrace X_j \rbrace$ be a family of pure $n$-dimensional simplicial complexes.
\begin{enumerate}
\item $\lbrace X_j \rbrace$ is called a family of coboundary expanders if there is a constant $\epsilon >0$ such that 
$$\forall 0 \leq k \leq n-1, \forall j, \epsilon_k (X_j ) \geq \epsilon.$$
\item  $\lbrace X_j \rbrace$ is called a family of cocycle expanders if there is a constant $\epsilon >0$ such that 
$$\forall 0 \leq k \leq n-1, \forall j, \widetilde{\epsilon}_k (X_j ) \geq \epsilon.$$
\end{enumerate} 
\end{definition}

\subsection{Topological expansion}
Let $X$ be an $n$-dimensional simplicial complex as before. Given a map $f : X^{(0)} \rightarrow \mathbb{R}^n$, a topological extension of $f$ is a continuous map $\widetilde{f} : X \rightarrow \mathbb{R}^n$ which coincides with $f$ on $X^{(0)}$.
\begin{definition}
A simplicial complex $X$ as above is said to have $c$-topological overlapping (with $1 \geq c>0$) if for every $f: X^{(0)} \rightarrow \mathbb{R}^n$ and every topological extension $\widetilde{f}$, there is a point $z \in \mathbb{R}^n$ such that 
$$\vert \lbrace \sigma \in X^{(n)} : z \in \widetilde{f} (\sigma)\rbrace \vert \geq c \vert X^{(n)} \vert .$$
In other words, this means that at least $c$ fraction of the images of $n$-simplices intersect at a single point. 
\end{definition}

A family of pure $n$-dimensional simplicial complexes $\lbrace X_j \rbrace$ is called a family of topological expanders, if there is some $c>0$ such that for every $j$, $X_j$ has  $c$-topological overlapping. \\
In \cite{Grom}, Gromov showed that any family of coboundary expanders (with respect to the homogeneous weight) is also a family of topological expanders, where the topological overlapping $c$ is a function of $\epsilon$  defined above.
 
\subsection{Local spectral expansion}

In \cite{Opp-LocalSpectral}, the author defined the concept of local spectral expansion:
\begin{definition}
A  pure $n$-dimensional simplicial complex is said to have $\lambda$-local spectral expansion, where $\frac{n-1}{n} < \lambda \leq 1$ if both of the following hold:
\begin{enumerate}
\item All the links of $X$ of dimension $>0$ are connected. 
\item All the $1$ dimensional links have spectral gap at least $\lambda$ (the $1$ dimensional links are the links of the $(n-2)$-simplices). 
\end{enumerate}
\end{definition}

This definition gives many results because theorem \ref{decent in links theorem} allows one to deduce much spectral data from the spectral gap of $1$-dimensional links.

\section{Localization inequalities in $\mathbb{F}_2$ coefficients}
Observe that for every $-1 \leq j \leq n-2$ and $\tau \in X^{(j)}$, the link $X_\tau$ is a weighted simplicial complex (with the weight $m_\tau$) and we can define the norm of it for every $\phi \in C^k (X_\tau, \mathbb{F}_2)$. We also denote $d_{\tau,k} : C^k (X_\tau, \mathbb{F}_2) \rightarrow C^{k+1} (X_\tau, \mathbb{F}_2)$ to be the differential (as before, we shall omit the index $k$ and denote only $d_\tau$).  Next, we'll define the concept of localization for $\phi \in C^k (X, \mathbb{F}_2)$:
\begin{definition}

Let $X$ be a pure $n$-dimensional weighted simplicial complex and let $-1 \leq j \leq n-1$. For $\tau \in X^{(j)}$, and $\phi \in C^k (X, \mathbb{F}_2)$ where $k-j-1 \geq 0$, define the localization of $\phi$ at $X_\tau$ to be $\phi_\tau \in  C^{k-j-1} (X_\tau, \mathbb{F}_2)$:
$$\forall \sigma \in X_\tau^{(k-j-1)}, \phi_\tau (\sigma ) = \phi (\tau \cup \sigma).$$
\end{definition}

\begin{remark}
When working with $C^k(X,\mathbb{R})$ (and not $C^k(X,\mathbb{F}_2)$) it is known that localization can be used to calculate norms of cochains and of differentials.  Below we establish such methods when working in $C^k(X,\mathbb{F}_2)$. As far as we know, the results below are new - one can find similar results in \cite{KKL}, proven only for Ramanujan complexes, but not for the general case. 
The reader should note that proposition \ref{F_2 norm from links proposition} and lemma \ref{F_2 differential norm and (k-1)-localization proposition} below have analogous results in $C^k(X,\mathbb{R})$ that are well known (appear for instance in \cite{BS}). The reader can compare these results to the results stated in \cite{Opp-LocalSpectral}[Lemma 4.4, Corollary 4.6].
\end{remark}

\begin{proposition}
\label{F_2 norm from links proposition}
Let $X$ be a pure $n$-dimensional weighted simplicial complex and let $-1 \leq j \leq n-1$. 
For every $k \geq j+1$ and every $\phi \in C^{k} (X, \mathbb{F}_2)$ we have that 
$${k+1 \choose j+1} \Vert \phi \Vert = \sum_{\tau \in X^{(j)}} \Vert \phi_\tau \Vert ,$$
where $\Vert \phi_\tau \Vert$ is the norm of $\phi_\tau$ in $X_\tau$.
\end{proposition}
 
\begin{proof}
Let $\phi \in C^{k} (X, \mathbb{F}_2$ as above. Then
\begin{dmath*}
\sum_{\tau \in X^{(j)}} \Vert \phi_\tau \Vert = \sum_{\tau \in X^{(j)}} \sum_{\sigma \in X_\tau^{(k-j-1)}, \phi_\tau (\sigma) =1} m_{\tau} (\sigma) = \sum_{\tau \in X^{(j)}} \sum_{\tau \cup \sigma  \in X^{(k)}, \phi (\tau \cup \sigma) =1} m (\tau \cup \sigma)  =\\
 \sum_{\tau \in X^{(j)}} \sum_{\gamma \in X^{(k)}, \tau \subset \gamma, \phi (\gamma) =1} m (\gamma) = \sum_{\gamma \in X^{(k)}, \phi (\gamma)= 1}  \sum_{\tau \in X^{(j)}, \tau \subset \gamma} m(\gamma) = \\
  \sum_{\gamma \in X^{(k)}, \phi (\gamma)= 1}  {k+1 \choose j+1} m(\gamma) = {k+1 \choose j+1} \Vert \phi \Vert .
\end{dmath*}
\end{proof}

\begin{proposition}
\label{F_2 differential norm and 0-localization proposition}
Let $1 \leq k \leq n-1$ and let $\phi \in C^{k} (X, \mathbb{F}_2)$. Then we have that: 
$$k \Vert d \phi \Vert + \Vert \phi \Vert \geq \sum_{\lbrace v \rbrace \in X^{(0)}} \Vert d_{\lbrace v \rbrace} \phi_{\lbrace v \rbrace} \Vert  .$$
\end{proposition}

\begin{proof}
Fix $\phi \in C^{k} (X, \mathbb{F}_2)$ and partition $X^{(k+1)}$ as follows: denote
$$T_0 = \lbrace \eta \in X^{(k+1)} : \forall \sigma \in X^{(k)}, \sigma \subset \eta \Rightarrow \phi (\sigma) =0 \rbrace ,$$
and for $i = 1,...,k+2$, denote
$$T_i = \lbrace \eta \in X^{(k+1)} : \exists! \sigma_1,...,\sigma_i \in X^{(k)}, \forall j, \sigma_j \subset \eta, \phi (\sigma_j) = 1 \rbrace .$$
Then $T_0,...,T_{k+2}$ is a disjoint partition of $X^{(k+1)}$. By the definition of the norm and the differential, we have that 
$$\Vert \phi \Vert = \sum_{j=0}^{k+2} (j) m(T_j).$$
$$\Vert d \phi \Vert = \sum_{i =1}^{\lceil \frac{k+2}{2} \rceil} m( T_{2i-1})  .$$
Note that for every $1 \leq i \leq k+2$ and for every $\eta \in T_i$, when choosing $v \in \eta$ there are exactly two possibilities: either $v \in \sigma_j$ for every $1 \leq j \leq i$ (where $\sigma_j$ are the $k$-simplices as in the definition of $T_i$), or there is a single $j_0$ such that $v \notin \sigma_{j_0}$.  Also note that 
$$\vert \sigma_1 \cap ... \cap \sigma_i \vert = k+2-i .$$
Therefore, for every  $1 \leq i \leq k+2$ and every $\eta \in T_i$ there are $k+2-i$ vertices $v$ such that $d_{\lbrace v \rbrace} \phi_{\lbrace v \rbrace} (\eta \setminus \lbrace v \rbrace) = d \phi (\eta)$ and $i$ vertices such that $d_{\lbrace v \rbrace} \phi_{\lbrace v \rbrace} (\eta \setminus \lbrace v \rbrace) = d \phi (\eta) + 1$ (when the addition is in $\mathbb{F}_2$). Therefore
$$\sum_{\lbrace v \rbrace \in X^{(0)}}  \Vert d_{\lbrace v \rbrace} \phi_{\lbrace v \rbrace} \Vert = \sum_{i =1}^{\lceil \frac{k+2}{2} \rceil} (k+2-(2i-1)) m( T_{2i-1} ) + \sum_{i =1}^{\lfloor \frac{k+2}{2} \rfloor} (2i) m( T_{2i} )  .$$
We finish by:
\begin{dmath*}
\sum_{\lbrace v \rbrace \in X^{(0)}}  \Vert  d_{\lbrace v \rbrace} \phi_{\lbrace v \rbrace} \Vert - \Vert \phi \Vert = \\
\sum_{i =1}^{\lceil \frac{k+2}{2} \rceil} (k+2-(2i-1)) m( T_{2i-1} ) + \sum_{i =1}^{\lfloor \frac{k+2}{2} \rfloor} (2i) m( T_{2i} )   - \sum_{j=0}^{k+2} (j) m(T_j) = \\
\sum_{i =1}^{\lceil \frac{k+2}{2} \rceil} (k+2-(2i-1) - (2i-1)) m( T_{2i-1} ) = \sum_{i =1}^{\lceil \frac{k+2}{2} \rceil} (k+4-4i) m( T_{2i-1} ) \leq k \Vert d \phi \Vert.
\end{dmath*}
\end{proof}

Using the above proposition we can get another result of this type:
\begin{lemma}
\label{F_2 differential norm and (k-1)-localization proposition}
Let $1 \leq k \leq n-1$ and let $\phi \in C^{k} (X, \mathbb{F}_2)$. Then we have that: 
$$\Vert d \phi \Vert  \geq \sum_{\tau \in X^{(k-1)}} \Vert d_\tau \phi_\tau \Vert  - k \Vert \phi \Vert.$$
\end{lemma}

\begin{proof}
We'll prove by induction. The case $k=1$ is proven in the above proposition. Let $k>1$.
For every $\lbrace v \rbrace \in X^{(0)}$ apply the induction assumption for each $\phi_{\lbrace v \rbrace} \in C^{k-1} (X_{\lbrace v \rbrace}, \mathbb{F}_2)$ to get 
\begin{dmath*}
\Vert d_{\lbrace v \rbrace} \phi_{\lbrace v \rbrace} \Vert \geq \left( \sum_{\eta \in X_{\lbrace v \rbrace}^{(k-2)}} \Vert (d_{\lbrace v \rbrace})_\eta (\phi_{\lbrace v \rbrace})_\eta \Vert \right) - (k-1) \Vert \phi_{\lbrace v \rbrace} \Vert = \\
\left( \sum_{\tau \in X^{(k-1)}, v \in \tau} \Vert d_\tau \phi_\tau \Vert \right) - (k-1) \Vert \phi_{\lbrace v \rbrace} \Vert.
\end{dmath*}
Summing on all vertices we get that 
\begin{dmath*}
\sum_{{\lbrace v \rbrace} \in X^{(0)}}  \Vert d_{\lbrace v \rbrace} \phi_{\lbrace v \rbrace} \Vert \geq \sum_{\lbrace v \rbrace \in X^{(0)}} \left( \left( \sum_{\tau \in X^{(k-1)}, v \in \tau} \Vert d_\tau \phi_\tau \Vert \right) - (k-1) \Vert \phi_{\lbrace v \rbrace} \Vert \right) = \\ \sum_{\tau \in X^{(k-1)}} k \Vert d_\tau \phi_\tau \Vert - (k+1)(k-1) \Vert \phi \Vert,
\end{dmath*}
where in the last equality we used proposition \ref{F_2 norm from links proposition}.
Next, by proposition \ref{F_2 differential norm and 0-localization proposition}:
$$k \Vert d \phi \Vert + \Vert \phi \Vert \geq \sum_{{\lbrace v \rbrace} \in X^{(0)}} \Vert d_{\lbrace v \rbrace} \phi_{\lbrace v \rbrace} \Vert  .$$
Therefore
$$k \Vert d \phi \Vert + \Vert \phi \Vert \geq \sum_{\tau \in X^{(k-1)}} k \Vert d_\tau \phi_\tau \Vert - (k+1)(k-1) \Vert \phi \Vert ,$$
which yields
$$\Vert d \phi \Vert  \geq \sum_{\tau \in X^{(k-1)}} \Vert d_\tau \phi_\tau \Vert  - k \Vert \phi \Vert,$$
and we are done.
\end{proof}

\section{Different notions of minimality}
In order to derive higher dimensional isoperimetric inequalities (see below), we must first discuss several notions of minimality of a cochain. The idea is that for $\phi \in C^k (X,\mathbb{F}_2)$ we want to measure the norm of $\phi$ up to $\psi \in B^k (X,\mathbb{F}_2)$. This can be done in several ways of optimality:
\begin{definition}

Let $0 \leq k \leq n-1$ and $\phi \in C^k (X,\mathbb{F}_2)$. 
\begin{enumerate}
\item $\phi$ is called minimal if for every $\varphi \in B^k (X,\mathbb{F}_2)$, we have that 
$$\Vert \phi \Vert \leq \Vert \phi - \varphi \Vert.$$
\item For $1 \leq k$, $\phi$ is called locally minimal if for every $\lbrace v \rbrace \in X^{(0)}$, $\phi_{\lbrace v \rbrace}$ is minimal in $X_{\lbrace v \rbrace}$ (this definition is taken from \cite{KKL}), i.e., if for every $\lbrace v \rbrace \in X^{(0)}$ and every $\varphi \in B^{k-1} (X_{\lbrace v \rbrace}, \mathbb{F}_2)$ we have that 
$$\Vert \phi_{\lbrace v \rbrace} \Vert \leq \Vert \phi_{\lbrace v \rbrace} - \varphi \Vert.$$ 
For $k=0$, $\phi \in C^0 (X,\mathbb{F}_2)$ is called locally minimal if it is minimal.
\item For $\varepsilon>0$, $k\geq 1$, $\phi \in C^k (X, \mathbb{F}_2)$ is called $\varepsilon$-locally minimal if for every $0 \leq j \leq k-1$ and every $\tau \in X^{(j)}$ we have that 
$$\forall \varphi \in B^{k-j-1} (X_\tau, \mathbb{F}_2), \Vert \phi_\tau \Vert \leq \Vert \phi_\tau - \varphi \Vert + \varepsilon m(\tau).$$
For $\phi \in  C^0 (X, \mathbb{F}_2)$, $\phi$ is called $\varepsilon$-locally minimal if 
$$\Vert \phi \Vert \leq \dfrac{(1+\varepsilon) m(X^{(0)})}{2}.$$
\end{enumerate}
\end{definition}

\begin{remark}
One can show that 
$$\phi \text{ is minimal } \Rightarrow \phi \text{ is locally minimal } \Rightarrow \phi \text{ is } \varepsilon-\text{locally minimal for every } \varepsilon>0.$$
One can also show that the reverse implications are false. We shall make no use of these facts, so the proofs are left to the reader. 
\end{remark}
\begin{remark}
\label{epsilon locally minimal implies epsilon very locally minimal}
It is not hard to see that for every $\phi \in C^k (X,\mathbb{F}_2)$, if $\phi$ is $\varepsilon$-locally minimal, then for every $\tau \in X^{(k-1)}$ we have that 
$$\Vert \phi_\tau \Vert \leq \dfrac{(1+\varepsilon)m(\tau)}{2}.$$
\end{remark}
The reason we consider $\varepsilon$-local minimality and not just local minimality (as in \cite{KKL}) is the following lemma:
\begin{lemma}
\label{inequality for varepsilon-locally min}
For every $\phi \in C^k (X, \mathbb{F}_2)$ and every $\varepsilon>0$  there is $\psi \in C^{k-1} (X,\mathbb{F}_2)$ such that $\phi - d \psi$ is $\varepsilon$-locally minimal and
$$\Vert \phi \Vert \geq \Vert \phi - d \psi \Vert + \varepsilon \Vert \psi \Vert. $$  
\end{lemma}

\begin{proof}
Fix $\varepsilon>0$. Let $\phi \in C^k(X, \mathbb{F}_2)$. If $\phi$ is $\varepsilon$-locally minimal we are done by taking $\psi \equiv 0$. Assume that $\phi$ is not $\varepsilon$-locally minimal. If $k=0$ and $\phi \in C^0 (X,\mathbb{F}_2)$, then 
$$\Vert \phi \Vert > \dfrac{(1+\varepsilon) m(X^{(0)})}{2}.$$
Take $\psi \in C^{-1} (X,\mathbb{F}_2)$ to be $\psi (\emptyset)=1$, then
\begin{dmath*}
\Vert \phi - d \psi \Vert < \dfrac{(1-\varepsilon) m(X^{(0)})}{2} = \dfrac{(1+\varepsilon) m(X^{(0)})}{2} - \varepsilon m(X^{(0)}) < \Vert \phi \Vert - \varepsilon m(X^{(0)}) = \Vert \phi \Vert - \varepsilon \Vert \psi \Vert,
\end{dmath*}
and we are done. Assume next that $k \geq 1$ and $\phi \in C^k (X,\mathbb{F}_2)$ is not $\varepsilon$-locally minimal. Then there is $0 \leq j \leq k-1$ and $\tau \in X^{(j)}$, $\psi_1' \in C^{k-j-2} (X_\tau, \mathbb{F}_2)$ such that 
$$\Vert \phi_\tau \Vert > \Vert \phi_\tau - d_\tau \psi_1' \Vert + \varepsilon m(\tau).$$ 
Define $\psi_1 \in C^{k-1} (X,\mathbb{F}_2)$ as 
$$\psi_1 (\eta) = \begin{cases}
\psi_1' (\eta \setminus \tau) & \tau \subset \eta \\
0 & \tau \not\subset \eta
\end{cases}.$$
Note the following: for $\sigma \in X^{(k)}$, if $\tau \not\subset \sigma$ then $d \psi_1 (\sigma ) =0$. Also note that for  $\sigma \in X^{(k)}$, if $\tau \subset \sigma$, then 
\begin{dmath*}
d \psi_1 (\sigma) = \sum_{\eta \in X^{(k-1)}, \eta \subset \sigma} \psi_1 (\eta) =  \sum_{\eta \in X^{(k-1)}, \eta \subset \sigma, \tau \subset \eta} \psi_1 (\eta) = \sum_{\eta \in X^{(k-1)}, \eta \subset \sigma, \tau \subset \eta} \psi_1' (\eta \setminus \tau) = d_\tau \psi_1' (\sigma \setminus \tau) = d_\tau \psi_1' (\sigma \setminus \tau).
\end{dmath*}
Therefore
\begin{dmath*}
\Vert \phi - d \psi_1 \Vert = \sum_{\sigma \in supp (\phi -  d \psi_1), \tau \not\subset \sigma} m(\sigma) + \sum_{\sigma \in supp (\phi -  d \psi_1), \tau \subset \sigma} m(\sigma) = \\ 
\sum_{\sigma \in supp (\phi), \tau \not\subset \sigma} m(\sigma)  + \sum_{\tau \subset \sigma, (\sigma \setminus \tau) \in supp (\phi_\tau -  d_\tau \psi_1')} m(\sigma) = \\
\sum_{\sigma \in supp (\phi), \tau \not\subset \sigma} m(\sigma) + \Vert \phi_\tau -  d_\tau \psi_1' \Vert < \\
 \sum_{\sigma \in supp (\phi), \tau \not\subset \sigma} m(\sigma) + \Vert \phi_\tau \Vert - \varepsilon m(\tau).
\end{dmath*}
Note that $\Vert \psi_1 \Vert = \Vert \psi_1' \Vert$ and that by corollary \ref{weight in l dim simplices} 
$$\Vert \psi_1' \Vert \leq \sum_{\theta \in X_\tau^{(k-j-2)}} m_\tau (\theta) = \sum_{\eta \in X^{(k-1)}, \tau \subset \eta} m(\eta) = \dfrac{1}{(k-1-j)!} m(\tau) \leq m(\tau).$$
Combined with the previous inequality, this yields
$$\Vert \phi \Vert \geq \Vert \phi - d \psi_1 \Vert + \varepsilon \Vert \psi_1 \Vert.$$
Define $\phi_1 = \phi - d \psi_1$. If $\phi_1$ is $\varepsilon$-locally minimal we are done by taking $\psi = \psi_1$. Otherwise repeat the same process to get $\psi_2 \in C^k (X,\mathbb{F}_2)$ and 
$$\Vert \phi_1 \Vert \geq \Vert \phi_1 - d \psi_2 \Vert + \varepsilon \Vert \psi_2 \Vert.$$
We continue this way until we get $\phi_l$ that is $\varepsilon$-locally minimal. Note that for every $i \leq l-1$ we have that  
$$\Vert \phi_i \Vert \geq \Vert \phi_{i+1} \Vert + \varepsilon \min \lbrace m(\tau) : \tau \in X^{(k-1)} \rbrace,$$
and therefore this procedure ends after finitely many steps. Define $\psi = \psi_1+...+\psi_l$ and note that $\phi_l = \phi - d \psi$ is $\varepsilon$-locally minimal and that
$$\Vert \phi - d \psi \Vert \leq \Vert \phi_1 - d (\psi_2+...+\psi_l) \Vert - \varepsilon \Vert \psi_1 \Vert \leq ... \leq \Vert \phi \Vert - \varepsilon ( \Vert \psi_1 \Vert + ... + \Vert \psi_l \Vert).$$
By the triangle inequality for the norm we get that
$$\Vert \phi \Vert \geq \Vert \phi - d \psi \Vert + \varepsilon \Vert \psi \Vert.$$
\end{proof}

\section{Criteria for topological overlapping}

To prove topological overlapping we'll relay on the following criterion taken from \cite{KKL}. In the theorem below, the norm $\Vert . \Vert$ refers to the norm with respect to the normalizes homogeneous weight $\overline{m_h}$ (also see remark below).

\begin{theorem}
\label{overlapping from KKL}
For constants $\mu > 0, \nu >0$, there is $c=c(n,\mu ,\nu)>0$ such that if $X$ is a finite simplicial complex (with the norm $\overline{m_h}$) satisfying:
\begin{enumerate}
\item For every $0 \leq k \leq n-1$, we have that
$$\mu_k (X) \leq \mu .$$
\item For every $0 \leq k \leq n-1$, we have that
$$\min \lbrace \Vert \phi \Vert : \phi \in Z^k (X,\mathbb{F}_2 ) \setminus B^k (X,\mathbb{F}_2 ) \rbrace \geq \nu.$$
\end{enumerate}
Then $X$ has the $c$-topological overlapping property.
\end{theorem}

\begin{remark}
\label{weights remark}
In other reference such as \cite{KKL}, the function that is used to define the norm is 
$$\forall \tau \in X^{(k)}, wt(\tau) = \dfrac{\vert \lbrace \eta \in X^{(n)} : \tau \subseteq \eta \rbrace \vert}{{n+1 \choose k+1} \vert X^{(n)} \vert} = \dfrac{(k+1)! \overline{m_h}(\tau)}{(n+1)!} ,$$
and the norm $\Vert . \Vert_{wt}$ is defined accordingly (as we defined it using $m$). We did not use this function, because since it is not a proper weight function (according to our definition) and therefore using it makes some summation procedures more cumbersome.
However, any inequality on the norm $\Vert . \Vert_{wt}$ can be easily transformed to an inequality about the norm $\Vert . \Vert$ defined by $\overline{m_h}$, by adding a multiplicative constant (dependent only on $n$ and on $k$). Therefore we allowed ourselves to stated the theorem using our norm and not $\Vert . \Vert_{wt}$.
\end{remark}

The idea (taken from \cite{KKL}) is that the conditions in the above theorem can be deduced from isoperimetric inequalities (apart from bounding $\mu_{n-1}$). 
\begin{lemma}
\label{isoperimetry implies conditions for topological overlap}
Let $X$ be a pure $n$-dimensional weighted simplicial complex. Assume there are constants $C_0 >0,...,C_{n-1} >0$ and $\varepsilon >0$ such that for  every $0 \leq k \leq n-1$ and for every $0 \neq \phi \in C^k (X,\mathbb{F}_2)$ we have that
$$\left( \phi \text{ is } \varepsilon \text{-locally minimal and } \Vert \phi \Vert \leq C_k m(X^{(k)}) \right) \Rightarrow \Vert d \phi \Vert >0.$$ 
Denote
$$\mu = \max \lbrace \dfrac{1}{\varepsilon}, \dfrac{1}{C_0}, \dfrac{2}{ C_1},...,\dfrac{n}{ C_{n-1}} \rbrace ,$$ 
$$\nu = \min \lbrace C_0,...,C_{n-1} \rbrace.$$ 
Then:
 \begin{enumerate}
\item For every $0 \leq k \leq n-2$, we have that $\mu_k (X) \leq \mu$ (with $\mu_k (X)$ being the k-th cofilling constant of $X$ defined above).
\item For every $0 \leq k \leq n-1$, we have that
$$\min \lbrace \Vert \phi \Vert : \phi \in Z^k (X,\mathbb{F}_2 ) \setminus B^k (X,\mathbb{F}_2 ) \rbrace \geq \nu m(X^{(k)}).$$
\end{enumerate}
\end{lemma}

\begin{proof}
We'll start by proving the second assertion. Fix $0 \leq k \leq n-1$ and let $\phi \in Z^k (X,\mathbb{F}_2)$. If $\Vert \phi \Vert \leq C_k m(X^{(k)})$, then by lemma \ref{inequality for varepsilon-locally min} there is $\psi \in C^{k-1} (X,\mathbb{F}_2)$ such that $\phi - d \psi$ is $\varepsilon$-locally minimal and 
$$\Vert \phi \Vert \geq \Vert \phi - d \psi \Vert + \varepsilon \Vert \psi \Vert.$$
In particular, $C_k m(X^{(k)}) \geq \Vert \phi \Vert \geq \Vert \phi - d \psi \Vert$, therefore by our assumptions we have that if $\phi - d \psi \neq 0$, then 
$$\Vert d (\phi - d \psi) \Vert >0,$$
but this is a contradiction to the fact that $\phi \in Z^k (X,\mathbb{F}_2)$. Therefore $\phi = d \psi$, which mean that $\phi \in B^k (X,\mathbb{F}_2)$. This yields that for every $\phi \in Z^k (X,\mathbb{F}_2) \setminus B^k (X,\mathbb{F}_2)$, we have that $\Vert \phi \Vert > C_k m(X^{(k)})$. \\
Next, we'll prove the first assertion of the lemma. Fix $0 \leq k \leq n-2$ and let $\phi \in B^{k+1} (X,\mathbb{F}_2)$. If $\Vert \phi \Vert \geq C_{k+1} m(X^{(k+1)})$, then we have that for every $\psi \in C^k (X,\mathbb{F}_2), d \psi = \phi$, that
$$\dfrac{\Vert \psi \Vert}{\Vert \phi \Vert} \leq \dfrac{m(X^{(k)})}{C_{k+1} m(X^{(k+1)})} = \dfrac{k+1}{C_{k+1}},$$
where the last equality is due to proposition \ref{weight of X^k compared to X^l}. Assume next that $\Vert \phi \Vert \leq C_{k+1} m(X^{(k+1)})$. By lemma \ref{inequality for varepsilon-locally min}, there is $\psi \in C^{k} (X,\mathbb{F}_2)$ such that $\phi - d \psi$ is $\varepsilon$-locally minimal and such that 
$$\Vert \phi \Vert \geq \Vert \phi - d \psi \Vert + \varepsilon \Vert \psi \Vert.$$
Therefore, $\Vert \phi - d \psi \Vert \leq C_{k+1} m(X^{(k+1)})$. Note that $\phi \in B^{k+1} (X,\mathbb{F}_2)$ and therefore $d (\phi - d \psi)=0$. Therefore, we can deduce that $\phi - d \psi =0$. Indeed, otherwise $\phi - d \psi$ is $\varepsilon$-locally minimal and $\Vert \phi - d \psi \Vert \leq C_{k+1} m(X^{(k+1)})$, which mean that $\Vert d (\phi - d \psi) \Vert >0$, which yields a contradiction. So we have that $\phi = d \psi$ and 
$$\Vert \phi \Vert \geq \Vert \phi - d \psi \Vert + \varepsilon \Vert \psi \Vert =\varepsilon \Vert \psi \Vert.$$
Therefore
$$\dfrac{\Vert \psi \Vert}{\Vert \phi \Vert} \leq \dfrac{1}{\varepsilon},$$
and we are done.    
\end{proof}

Using the above lemma combined with the criterion for topological overlap stated above, we can deduce the following:

\begin{lemma}
\label{topological overlap of the n-1 skeleton based on isoperimetric ineq}
Let $1 \leq l \leq n-1,  M \geq 1, C_0 >0,...,C_{l} >0, \varepsilon >0$. There is a constant $c = c(M,C_0,...,C_{l}, \varepsilon)$ such that for every simplicial complex $X$ of dimension $n>1$, if 
\begin{enumerate}
\item For  every $0 \leq k \leq l$ and for every $0 \neq \phi \in C^k (X,\mathbb{F}_2)$ we have that
$$\left( \phi \text{ is } \varepsilon \text{-locally minimal and } \Vert \phi \Vert \leq C_k \overline{m_h}(X^{(k)}) \right) \Rightarrow \Vert d \phi \Vert >0.$$ 
\item We have that
$$\dfrac{\sup_{\sigma \in X^{(l)}} \vert \lbrace \eta \in X^{(n)} : \sigma \subset \eta \rbrace \vert }{\inf_{\sigma \in X^{(l)}} \vert \lbrace \eta \in X^{(n)} : \sigma \subset \eta \rbrace \vert } \leq M .$$
\end{enumerate} 
Then the $l$-skeleton of $X$ has $c$-topological overlapping property.
\end{lemma}

\begin{proof}
In lemma \ref{isoperimetry implies conditions for topological overlap}, we showed that there are $\mu = \mu (C_0,...,C_{l}, \varepsilon)$, $\nu = \nu ((C_0,...,C_{l})$ such that 
  \begin{enumerate}
\item For every $0 \leq k \leq l-1$, we have that $\mu_k (X) \leq \mu$ (with $\mu_k (X)$ being the k-th cofilling constant of $X$ defined above).
\item For every $0 \leq k \leq l-1$, we have that
$$\min \lbrace \Vert \phi \Vert : \phi \in Z^k (X,\mathbb{F}_2 ) \setminus B^k (X,\mathbb{F}_2 ) \rbrace \geq \nu \overline{m_h} (X^{(k)}) = \nu \dfrac{(n+1)!}{(k+1)!} .$$
\end{enumerate}
However, one should note that we cannot apply theorem \ref{overlapping from KKL} to the $l$-skeleton yet, since we have to deal with the following issue: all the inequalities states above refer to the norm $\overline{m_h}$ on $X$ that is based on the $n$-dimensional simplices. To apply theorem \ref{overlapping from KKL} on the $l$-skeleton, we shall need similar inequalities when all the norms are computed with respect to the norm $\overline{m_{h,l}}$ defined using the $l$ simplices as:
$$\forall 0 \leq k \leq l, \forall \tau \in X^{(k)}, \overline{m_{h,l}} (\tau ) = \dfrac{(l-k)! \vert \lbrace \eta \in X^{(l)} : \tau \subseteq \eta \rbrace \vert}{\vert X^{(l)} \vert} .$$
Therefore, we'll need to compare the norm calculated by $\overline{m_{h,l}}$ to the norm calculated by $\overline{m_{h}}$.  
Denote the norm with respect to $\overline{m_{h,l}}$ as $\Vert . \Vert_{l}$. Also denote
$$M_1 = \inf_{\sigma \in X^{(l)}} \vert \lbrace \eta \in X^{(n)} : \sigma \subset \eta \rbrace \vert ,$$
$$M_2 = \sup_{\sigma \in X^{(l)}} \vert \lbrace \eta \in X^{(n)} : \sigma \subset \eta \rbrace \vert.$$
Then we have that for every $\sigma \in X^{(l)}$ that 
$$M_1 \leq \vert \lbrace \eta \in X^{(n)} : \sigma \subset \eta \rbrace \vert \leq M_2,$$
and
$$ \dfrac{M_1}{{n+1 \choose l+1}}\vert X^{(l)} \vert \leq \vert X^{(n)} \vert \leq  \dfrac{M_2}{{n+1 \choose l+1}} \vert X^{(l)} \vert.$$
Therefore we have for every $\sigma \in X^{(l)}$ that 
$$\dfrac{(n-l)! {n+1 \choose l+1}}{M} \overline{m_{h,l}} (\sigma) =\dfrac{(n-l)! M_1}{\frac{M_2}{{n+1 \choose l+1}} \vert X^{(l)} \vert} \leq \dfrac{(n-l)! \vert \lbrace \eta \in X^{(n)} : \sigma \subset \eta \rbrace \vert}{\vert X^{(n)} \vert} = \overline{m_{h}} (\sigma ).$$
Similarly for every $\sigma \in X^{(l)}$
$$\overline{m_{h}} (\sigma ) \leq \dfrac{(n-l)! {n+1 \choose l+1}  M_2}{M_1 \vert X^{(l)} \vert} = (n-l)! {n+1 \choose l+1} M \overline{m_{h,l}} (\sigma).$$
Therefore, by the definition of the weight function, we have that for every $0 \leq k \leq l$ and every $\tau \in X^{(k)}$ the following 
$$\dfrac{(n-l)! {n+1 \choose l+1}}{M} \overline{m_{h,l}} (\tau)   \leq  \overline{m_{h}} (\tau) \leq (n-l)! {n+1 \choose l+1} M \overline{m_{h,l}} (\tau) .$$
This in turn yields that for every $0 \leq k \leq l$ and every $\phi \in C^k (X,\mathbb{F}_2)$ we have that
$$\dfrac{(n-l)! {n+1 \choose l+1}}{M}  \Vert \phi \Vert_{l} \leq \Vert \phi \Vert  \leq (n-l)! {n+1 \choose l+1} M\Vert \phi \Vert_{l} .$$
Therefore any inequality stated in the usual norm $\Vert . \Vert$ of $X$ can be transformed to an inequality in the "$l$-skeleton norm" (the constants may change as $M$ changes). Explicitly, let $X'$ be the $l$-skeleton of $X$, with the norm $\Vert . \Vert_{l}$, then  with $\mu, \nu$ as above we have that
\begin{enumerate}
\item For every $0 \leq k \leq l-1$, we have that $\mu_k (X') \leq M^2 \mu$.
\item For every $0 \leq k \leq l-1$, we have that
$$\min \lbrace \Vert \phi \Vert_{l} : \phi \in Z^k (X',\mathbb{F}_2 ) \setminus B^k (X',\mathbb{F}_2 ) \rbrace \geq \dfrac{1}{M}  \dfrac{(l+1)!}{(k+1)!} \nu .$$
\end{enumerate}
Therefore we are done by applying theorem \ref{overlapping from KKL} to $X'$.
\end{proof}

\section{Isoperimetric inequalities}
In this section we shall prove the main result of this paper. The main idea of this result and its proof are taken from \cite{KKL}.
Following \cite{KKL}, we define the notion of thick and thin:
\begin{definition}
Let $0 < \delta <1 , 0 < r \leq 1$ and $X$ be a pure $n$-dimensional weighted simplicial complex. Define the following: 
\begin{enumerate}
\item $\phi \in C^0 (X,\mathbb{F}_2)$ will be called $\delta$-thin if
$$\Vert \phi \Vert \leq \delta m(X^{(0)}).$$
Otherwise, we shall call $\phi$ $\delta$-thick.
\item For $k>0$ and $\phi \in C^k (X,\mathbb{F}_2)$, we shall call $\tau \in X^{(k-1)}$, $\delta$-thin, is $\phi_\tau \in  C^0 (X_\tau,\mathbb{F}_2)$ if $\delta$-thin, i.e., if
$$\Vert \phi_\tau \Vert \leq \delta m(\tau) .$$
Otherwise, we shall call $\tau$ $\delta$-thick. Denote 
$$A_\delta = \lbrace \tau \in X^{(k-1)} : \tau \text{ is } \delta \text{-thin} \rbrace.$$
\item For $k>0$ and $\phi \in C^k (X,\mathbb{F}_2)$, we shall call $\phi$ $(r, \delta)$-thin, if
$$\sum_{\tau \in A_\delta} \Vert \phi_\tau \Vert \geq r (k+1) \Vert \phi \Vert  .$$
Otherwise, $\phi$ will be called $(r, \delta)$-thick.
\end{enumerate}
\end{definition}

\begin{lemma}
\label{thin has laplacian inequality}
Let $X$ be a pure $n$-dimensional weighted simplicial. Let $0 < \varepsilon <1 , 0 < \delta <\frac{1}{2} , 0 < r \leq 1$. Denote 
$$\lambda_0 = \lambda (X),$$
$$k \geq 1, \lambda_{k} = \min_{\tau \in X^{(k-1)}} \lambda (X_\tau),$$
(see \ref{Links and spectral gaps subsection} to recall the definition of $\lambda (X_\tau)$ and some facts about it). \\
\begin{enumerate}
\item If $\phi \in C^{0} (X,\mathbb{F}_2)$ is $\delta$-thin, then
$$\Vert d \phi \Vert \geq \lambda_0 (1-\delta) \Vert \phi \Vert.$$
\item For $k>0$, if $\phi \in C^{k} (X,\mathbb{F}_2)$ is $\varepsilon$-locally minimal and $(r,\delta)$-thin, then
$$\Vert d \phi \Vert  \geq   \left(  \left(   \dfrac{r+1}{2}   - \delta  - \dfrac{\varepsilon}{2} \right)  \lambda_k (k+1)  - k \right) \Vert \phi \Vert  .$$
\end{enumerate}

\end{lemma}

\begin{proof}
\begin{enumerate}
\item Let $\phi \in C^{0} (X,\mathbb{F}_2)$ that is $\delta$-thin. Note that by definition
$$\Vert d \phi \Vert = m (supp (\phi), X^{(0)} \setminus supp (\phi) ),$$
and 
$$\Vert \phi \Vert = m(supp (\phi) ) .$$
Therefore, the assumption that $\phi$ is $\delta$-thin, yields that 
$$ m(supp (\phi) ) \leq \delta m(X^{(0)}),$$
and equivalently that 
$$ m(X^{(0)} \setminus supp (\phi) ) \geq (1-\delta) m(X^{(0)}).$$
By proposition \ref{Cheeger for weighted graphs proposition}, we have that 
$$m (supp (\phi), X^{(0)} \setminus supp (\phi) ) \geq \lambda_0 \dfrac{m(supp (\phi)) m(X^{(0)} \setminus supp (\phi) )}{m(X^{(0)})} \geq \lambda_0 (1-\delta) \Vert \phi \Vert ,$$
and we are done.
\item  For $k>0$, let $\phi \in C^{k} (X,\mathbb{F}_2)$ be $\varepsilon$-locally minimal and $(r,\delta)$-thin. Note that for every $\tau \in X^{(k-1)}$, $X_\tau$ is a weighted graph with
$$\lambda (X_\tau) \geq \lambda_k.$$
By the assumption that $\phi$ is $\varepsilon$-locally minimal, we have for every $\tau \in X^{(k-1)}$ that $\phi_\tau$ is $\frac{1+\varepsilon}{2}$-thin (by remark \ref{epsilon locally minimal implies epsilon very locally minimal}). Therefore by result above for the case $k=0$, we have for every $\tau \in X^{(k-1)}$ that 
$$\Vert d_\tau \phi_\tau \Vert \geq \lambda_k \dfrac{1-\varepsilon}{2} \Vert \phi_\tau \Vert  .$$
For $\tau \in A_\delta$, we have by the result for $k=0$ that 
$$\Vert d_\tau \phi_\tau \Vert \geq \lambda_k (1-\delta) \Vert \phi_\tau \Vert  .$$
Therefore
\begin{dmath*}
\sum_{\tau \in X^{(k-1)}} \Vert d_\tau \phi_\tau \Vert = \sum_{\tau \in A_\delta} \Vert d_\tau \phi_\tau \Vert + \sum_{\tau \in X^{(k-1)} \setminus A_\delta} \Vert d_\tau \phi_\tau \Vert \geq \\
  \lambda_k \left( (1-\delta) \sum_{\tau \in A_\delta} \Vert \phi_\tau \Vert +  \dfrac{1-\varepsilon}{2} \sum_{\tau \in X^{(k-1)} \setminus A_\delta}  \Vert \phi_\tau \Vert \right) = \\
 \lambda_k \left( \dfrac{1}{2}  \sum_{\tau \in X^{(k-1)}} \Vert \phi_\tau \Vert  +  (\dfrac{1}{2}-\delta)  \sum_{\tau \in A_\delta} \Vert \phi_\tau \Vert  - \dfrac{\varepsilon}{2}  \sum_{\tau \in X^{(k-1)} \setminus A_\delta} \Vert \phi_\tau \Vert \right) = \\
  \lambda_k  \left( \dfrac{k+1}{2}  \Vert \phi \Vert +  (\dfrac{1}{2}-\delta) \sum_{\tau \in A_\delta}  \Vert \phi_\tau \Vert  - \dfrac{\varepsilon}{2}  \sum_{\tau \in X^{(k-1)} \setminus A_\delta} \Vert \phi_\tau \Vert \right).
\end{dmath*}
Next, by the assumption that $\phi$ is $(r,\delta)$-thin we get by the above proposition that
\begin{dmath*}
{(\dfrac{1}{2}-\delta) \sum_{\tau \in A_\delta}   \Vert \phi_\tau \Vert  - \dfrac{\varepsilon}{2}  \sum_{\tau \in X^{(k-1)} \setminus A_\delta} \Vert \phi_\tau \Vert \geq} \\ \left(  r (k+1) (\dfrac{1}{2}-\delta) -  (1-r)(k+1) \dfrac{\varepsilon}{2} \right) \Vert \phi \Vert.
\end{dmath*}
Therefore we showed that 
\begin{dmath*}
\sum_{\tau \in X^{(k-1)}} \Vert d_\tau \phi_\tau \Vert \geq \left(  \dfrac{1}{2} + r (\dfrac{1}{2}-\delta) -  (1-r) \dfrac{\varepsilon}{2} \right) \lambda_k (k+1) \Vert \phi \Vert = \\
 \left(   \dfrac{r+1}{2}   -r \delta  - (1-r) \dfrac{\varepsilon}{2} \right) \lambda_k (k+1) \Vert \phi \Vert .
\end{dmath*}

By lemma \ref{F_2 differential norm and (k-1)-localization proposition}, we have that
$$\Vert d \phi \Vert  \geq \sum_{\tau \in X^{(k-1)}} \Vert d_\tau \phi_\tau \Vert  - k \Vert \phi \Vert.$$
Therefore
$$\Vert d \phi \Vert  \geq   \left(  \left(   \dfrac{r+1}{2}   -r \delta  - (1-r) \dfrac{\varepsilon}{2} \right)  \lambda_k (k+1)  - k \right) \Vert \phi \Vert  .$$
From the fact that $0 < r \leq 1$, this yields that 
$$\Vert d \phi \Vert  \geq   \left(  \left(   \dfrac{r+1}{2}   - \delta  - \dfrac{\varepsilon}{2} \right)  \lambda_k (k+1)  - k \right) \Vert \phi \Vert  .$$
\end{enumerate}
\end{proof}

A simple corollary of the above lemma is the following isoperimetric inequality for the case $k=0$:
\begin{corollary}
\label{k=0 isoperimetric inequality}
For $X$ that is a pure $n$-dimensional weighted simplicial complex with $n \geq 1$, denote 
$$\lambda_0 = \lambda (X).$$
For every $X$ as above, if $\lambda_0 >0$ then for every $1> \varepsilon >0$, we have for every $0 \neq \phi \in C^0 (X, \mathbb{F}_2)$ that
$$\left( \phi \text{ is } \varepsilon \text{-locally minimal} \right) \Rightarrow \Vert d \phi \Vert \geq \lambda_0 \dfrac{1-\varepsilon}{2} \Vert \phi \Vert >0.$$ 
\end{corollary}

\begin{proof}
Fix $1> \varepsilon >0$. By definition we have that for every $\phi \in C^0 (X,\mathbb{F}_2)$, if $\phi$ is $\varepsilon$-locally minimal, then
$$\Vert \phi \Vert \leq \dfrac{1+\varepsilon}{2} m(X^{(0)}).$$
By the above lemma, we have that  
$$\Vert d \phi \Vert \geq \lambda_0 (1-\frac{\varepsilon}{2}) \Vert \phi \Vert >0,$$
and we are done.
\end{proof}

Next, we shall prove the following:
\begin{lemma}
\label{small implies thin lemma}
For $X$  a pure $n$ dimensional weighted simplicial complex of dimension $n >1$ denote 
$$\lambda_0 = \lambda (X).$$
For every $\delta >0, \varepsilon_1 >0$, there are constants $0< C_1' = C_1' (\delta, \varepsilon_1) <1, \theta_1' = \theta_1' (\delta, \varepsilon_1 ) <1$, such that if $X$ is as above with
$$\lambda_0 \geq \theta_1',$$
then for every $\phi \in C^1 (X,\mathbb{F}_2)$, if $\Vert \phi \Vert \leq C_1' m(X^{(1)})$ then:
\begin{enumerate}
\item 
$$2 \varepsilon_1 \Vert \phi \Vert \geq \sum_{\lbrace u, v\rbrace \in X^{(1)}, \lbrace u \rbrace \notin A_\delta, \lbrace v \rbrace \notin A_\delta} m(\lbrace u, v\rbrace) .$$
\item  $\phi$ is $(\frac{1}{2}-\varepsilon_1 , \delta)$-thin.
\end{enumerate}
\end{lemma}

\begin{proof}
Fix $\delta>0, \varepsilon_1>0$.  Choose 
$$C_1' = \delta^2 \varepsilon_1 ,$$
$$\theta_1' = 1-\delta \varepsilon_1 .$$
Assume that $\lambda_0 \geq \theta_1'$ and let $\phi \in C^1 (X,\mathbb{F}_2)$ such that $\Vert \phi \Vert \leq C_1' m(X^{(1)})$. In the case $k=1$, $A_\delta \subset X^{(0)}$. We'll denote 
$$R^1 = A_\delta , S^1 = X^{(0)} \setminus S^1.$$
(these notations are in order to make the proof of this lemma similar to the proof of theorem \ref{k=2 alternative theorem} below). \\
Note that for every $\lbrace v \rbrace \in S^1$ we have that 
$$\Vert \phi_{\lbrace v \rbrace} \Vert \geq \delta m(v) .$$
By proposition \ref{F_2 norm from links proposition}, we have that
$$2 \Vert \phi \Vert \geq \sum_{\lbrace v \rbrace \in S^1} \Vert \phi_{\lbrace v \rbrace} \Vert  \geq \sum_{\lbrace v \rbrace \in S^1} \delta m(v) = \delta m(S^1) .$$
Therefore 
\begin{equation}
\label{ineq1}
\dfrac{2}{\delta} \Vert \phi \Vert \geq m(S^1), 
\end{equation}
and 
$$2 C_1' m(X^{(1)}) \geq \delta m(S^1) .$$
By the choice of $C_1'$ and since $m(X^{(0)}) = 2 m(X^{(1)})$ we get that 
$$\delta \varepsilon_1 \geq \dfrac{m(S^1)}{m (X^{(0)})}.$$
Equivalently,
\begin{equation}
\label{ineq2} 
1 - \delta \varepsilon_1  \leq \dfrac{m(R^1)}{m (X^{(0)})}.
\end{equation}
Recall that by proposition \ref{Cheeger for weighted graphs proposition}, we have that 
$$\dfrac{m(S^1)}{2} \left( 1 - \lambda_0 \dfrac{m(R^1)}{m (X^{(0)})} \right) \geq m(S^1, S^1) .$$
Combining the above inequality with \eqref{ineq1}, \eqref{ineq2} and the choice of $\theta_1'$ we get that
$$\dfrac{1}{\delta} \Vert \phi \Vert  \left( 1 - (1-\delta \varepsilon_1)^2 \right) \geq m(S^1, S^1) .$$
This yields that 
\begin{equation}
\label{ineq3}
2 \varepsilon_1 \Vert \phi \Vert  \geq m(S^1, S^1) .
\end{equation}
Next, for $i=0,1,2$ denote the following sets 
$$K_i^1 = \lbrace \sigma \in X^{(1)} : \vert \lbrace \lbrace v \rbrace \in S^1 : \lbrace v \rbrace \subset \sigma \rbrace  \vert =i \rbrace .$$
Note that $X^{(1)} = K_0^1 \cup K_1^1 \cup K_2^1$ and all the above sets are disjoint. With this notation we reinterpret \eqref{ineq3} as
$$2 \varepsilon_1 \Vert \phi \Vert \geq m(K_2^1) \geq m(K_2^1 \cap supp (\phi)) .$$
This proves the first assertion in the lemma. The above inequality yields that
$$  (1- 2\varepsilon_1)  \Vert \phi \Vert \leq m(K_0^1 \cap supp (\phi)) + m(K_1^1 \cap supp (\phi)) .$$
Also note that 
$$\sum_{\lbrace v \rbrace \in R^1} \Vert \phi_{\lbrace v \rbrace} \Vert = 2 m(K_0^1 \cap supp (\phi)) + m(K_1^1 \cap supp (\phi)) .$$
Therefore 
$$ (1- 2 \varepsilon_1)  \Vert \phi \Vert \leq 2 \sum_{\lbrace v \rbrace \in R^1} \Vert \phi_{\lbrace v \rbrace} \Vert ,$$
which yields 
$$2  (\dfrac{1}{2}- \varepsilon_1)  \Vert \phi \Vert \leq \sum_{\lbrace v \rbrace \in R^1} \Vert \phi_{\lbrace v \rbrace} \Vert,$$
as needed.
\end{proof}

Combining the two lemmas above we get the following isoperimetric inequality :
\begin{theorem}
\label{k=1 isoperimetric inequality theorem}
For $X$ that is a pure $n$-dimensional weighted simplicial complex with $n>1$, denote 
$$\lambda_0 = \lambda (X),$$
$$\lambda_{1} = \min_{\lbrace v \rbrace \in X^{(0)}} \lambda (X_{\lbrace v \rbrace}).$$
There are constants $\varepsilon  >0,  \theta_1 = \theta_1 <1, C_1 = C_1 > 0$ such that for every $X$ as above we have that if $\min \lbrace \lambda_0, \lambda_1 \rbrace \geq \theta_1$, then for every $0 \neq \phi \in C^1 (X,\mathbb{F}_2)$
$$\left( \phi \text{ is } \varepsilon \text{-locally minimal and } \Vert \phi \Vert \leq C_1 m(X^{(1)}) \right) \Rightarrow \Vert d \phi \Vert \geq \dfrac{1}{4} \Vert \phi \Vert >0.$$
\end{theorem}

\begin{proof}
Take 
$$\delta = \varepsilon = \dfrac{1}{16}, \varepsilon_1 = \dfrac{1}{32}.$$
With the above choice let $\theta_1' = \theta_1' (\delta, \varepsilon_1), C_1' = C_1' (\delta, \varepsilon_1)$ as in lemma \ref{small implies thin lemma}. Next take $C_1 = C_1' (\delta, \varepsilon_1)$, 
$$\theta_1 = \max \left\lbrace \theta_1', \dfrac{\frac{1}{4}}{\frac{1}{2} - \frac{7}{32}} \right\rbrace .$$
Let $X$ with $\min \lbrace \lambda_0, \lambda_1 \rbrace \geq \theta_1$ and let $\phi \in C^1 (X,\mathbb{F}_2)$ such that $\phi$ is $\varepsilon$-locally minimal and $\Vert \phi \Vert \leq C_1 m(X^{(1)})$. By lemma \ref{small implies thin lemma}, we have that $\phi$ is $(\frac{1}{2}-\frac{1}{32},\delta)$-thin. By lemma \ref{thin has laplacian inequality} for $k=1$, we get that 
$$\Vert d \phi \Vert  \geq   \left(   \lambda_1 (1+\frac{1}{2}  -\dfrac{1}{32} - 2\delta - \varepsilon) -1  \right) \Vert \phi \Vert  .$$
By the choice of $\delta, \varepsilon, \varepsilon_1$ this yields
$$\Vert d \phi \Vert  \geq   \left(   \lambda_1 (1+\frac{1}{2} -\frac{1}{32}-  \dfrac{1}{8}- \dfrac{1}{16} -1  \right) \Vert \phi \Vert  .$$
After simplifying, we get that 
$$\Vert d \phi \Vert  \geq   \left(   \lambda_1 (\dfrac{1}{2}- \dfrac{7}{32})  \right) \Vert \phi \Vert  .$$
To finish, recall that by the choice of $\theta_1$, we have that $\lambda_1 \geq \frac{ \frac{1}{4}}{\frac{1}{2} - \frac{7}{32}}$ and therefore
$$\Vert d \phi \Vert \geq \dfrac{1}{4} \Vert \phi \Vert .$$
\end{proof}

\begin{remark}
Our choice of $\frac{1}{4}$ in the above theorem is arbitrary: for every $\varepsilon_1'>0$, one can find $C_1,\theta_1, \varepsilon$, such that the formulation of the above theorem reads
$$\left( \phi \text{ is } \varepsilon \text{-locally minimal and } \Vert \phi \Vert \leq C_1 m(X^{(1)}) \right) \Rightarrow \Vert d \phi \Vert \geq \left(\dfrac{1}{2} - \varepsilon_1' \right) \Vert \phi \Vert.$$
\end{remark}

\begin{remark}
The above theorem provides an isoperimetric inequality for the case $k=1$ (i.e., for $\phi \in C^1 (X,\mathbb{F}_2)$ under certain conditions). Note that this result does not depend on anything other than the spectral properties of the simplicial complex and can be deduced from large enough local spectral expansion. At this point, we do not know how to prove isoperimetric inequality in the $k=2$ cases strictly from spectral gap considerations (it is our hope to do so in the future). Below, we will show the isoperimetric inequality for the case $k=2$ under further assumptions. 
\end{remark}

\begin{theorem}
\label{k=2 alternative theorem}
For $X$ that is a pure $n$-dimensional weighted simplicial complex with $n>2$, denote 
$$\lambda_0 = \lambda (X),$$
$$\lambda_{1} = \min_{\lbrace v \rbrace \in X^{(0)}} \lambda (X_{\lbrace v \rbrace}),$$
$$\lambda_{2} = \min_{\tau \in X^{(1)}} \lambda (X_\tau).$$
For every $1 \geq \varepsilon_2 >0, \delta >0$ there are constants $ \theta_2' = \theta_2' (\varepsilon_2, \delta)<1, C_1' = C_1' (\varepsilon_2, \delta)> 0, C_2' = C_2' (\varepsilon_2, \delta)> 0$ such that for every $X$ as above and $\phi \in C^2 (X,\mathbb{F}_2)$ we have that if:
\begin{enumerate}
\item  $\min \lbrace \lambda_0, \lambda_1, \lambda_2 \rbrace \geq \theta_2'$.
\item $\Vert \phi \Vert \leq C_2' m(X^{(2)})$.
\end{enumerate}
Then one of the following holds:
\begin{enumerate}
\item There is a set $S^2 \subset X^{(0)}$ such that 
$$\forall \lbrace v \rbrace \in S^2, \Vert \phi_{\lbrace v \rbrace} \Vert \geq C_1' m(X_{\lbrace v \rbrace}^{(1)}) ,$$
$$ \sum_{\lbrace v \rbrace \in S^2} \Vert \phi_{\lbrace v \rbrace} \Vert \geq \dfrac{9}{20} \Vert \phi \Vert   ,$$
and
$$\Vert d \phi \Vert \geq \sum_{\lbrace v \rbrace \in S^2} \Vert d_{\lbrace v \rbrace} \phi_{\lbrace v \rbrace} \Vert -\dfrac{ 11 \varepsilon_2}{9} \sum_{\lbrace v \rbrace \in S^2} \Vert \phi_{\lbrace v \rbrace} \Vert.$$
\item $\phi$ is $(\frac{1}{3} + \frac{\varepsilon_2}{15}, \delta)$-thin.
\end{enumerate}
\end{theorem}

\begin{proof}
Fix $1 \geq \varepsilon_2 >0 , \delta >0$. Denote $\varepsilon_1 = \frac{\varepsilon_2}{60}$ and let $C_1' = C_1' (\varepsilon_1 ,\delta)$, $\theta_1' = \theta_1' (\varepsilon_1 , \delta)$ be the constants from lemma \ref{small implies thin lemma} above. Choose
$$\theta_2' = \max \left\lbrace \theta_1', 1- \dfrac{C_1' \varepsilon_2}{60} \right\rbrace,$$
$$C_2' = \dfrac{(C_1')^2 \varepsilon_2}{60} ,$$
Let $X$ as above with $\min \lbrace \lambda_0, \lambda_1, \lambda_2 \rbrace \geq \theta_2'$. Let $\phi \in C^2 (X, \mathbb{F}_2)$ such that $\Vert \phi \Vert \leq C_2' m(X^{(2)})$. Partition the vertices of $X^{(0)}$ as follows:
$$R^2 = \lbrace \lbrace v \rbrace \in X^{(0)} : \Vert \phi_{\lbrace v \rbrace} \Vert \leq C_1' m(X_{\lbrace v \rbrace}^{(1)} ) \rbrace,$$
$$S^2 = X^{(0)} \setminus R^2 .$$
Note that for every $v \in X^{(0)}$, we have that 
$$ \lambda_0 (X_{\lbrace v \rbrace}) = \lambda (X_{\lbrace v \rbrace}) \geq \lambda_1 (X).$$
Also note that 
\begin{dmath*}
3 C_2' m(X^{(2)}) \geq 3 \Vert \phi \Vert \geq \sum_{\lbrace v \rbrace \in S^2} \Vert \phi_{\lbrace v \rbrace} \Vert \geq \sum_{\lbrace v \rbrace \in S^2} C_1' m(X_{\lbrace v \rbrace}^{(1)}) = \sum_{\lbrace v \rbrace \in S^2} C_1' \dfrac{1}{2} m(v) = \dfrac{1}{2} C_1' m(S^2).
\end{dmath*}
Therefore
$$\dfrac{6}{C_1'} \Vert \phi \Vert \geq m(S^2),$$
$$\dfrac{C_1' \varepsilon_2}{60} = \dfrac{C_2'}{C_1'} \geq \dfrac{m(S^2)}{6 m(X^{(2)})} = \dfrac{m(S)}{m(X^{(0)})}.$$
The last inequality yields that 
$$\dfrac{m(R^2)}{m(X^{(0)})} \geq 1 -\dfrac{C_1' \varepsilon_2}{60} .$$
As in the proof of lemma \ref{small implies thin lemma}, we imply proposition \ref{Cheeger for weighted graphs proposition} and get that 
$$\dfrac{m(S^2)}{2} (1- \dfrac{m(R^2)}{m(X^{(0)})}) \geq m(S^2,S^2).$$
Therefore by the inequalities above and the choice of $\theta_2'$, we get that
$$\dfrac{3}{C_1'} \Vert \phi \Vert (1- (1- \dfrac{C_1' \varepsilon_2}{60})^2) \geq m(S^2,S^2) .$$
Therefore
$$\dfrac{\varepsilon_2}{10} \Vert \phi \Vert  \geq m(S^2,S^2).$$
For $i=0,1,2,3$ denote 
$$K_i^2 = \lbrace \sigma \in X^{(2)} : \vert \lbrace \lbrace v \rbrace \in S^2 : \lbrace v \rbrace \subset \sigma  \rbrace \vert = i \rbrace.$$
$\lbrace K_i^2 \rbrace_{i=0}^3$ is a partition of $X^{(2)}$ and therefore
$$\Vert \phi \Vert = \sum_{i=0}^3 m(K_i^2 \cap supp (\phi)).$$ 
Note that 
\begin{dmath*}
m(K_2^2) + 3 m(K_3^2) = \sum_{\sigma \in K_2^2} m(\sigma) +  \sum_{\sigma \in K_3^2} 3 m(\sigma) = \sum_{\sigma \in K_2^2} \left( \sum_{\lbrace u,v \rbrace \in  X^{(1)}, \lbrace u,v \rbrace \subset \sigma, \lbrace u \rbrace \in S^2, \lbrace v \rbrace \in S^2} m(\sigma) \right) +  \sum_{\sigma \in K_3^2} \left( \sum_{\lbrace u,v \rbrace \in  X^{(1)}, \lbrace u,v \rbrace \subset \sigma, \lbrace u \rbrace \in S^2, \lbrace v \rbrace \in S^2} m(\sigma) \right)=    
\sum_{\lbrace u,v \rbrace \in  X^{(1)}, \lbrace u \rbrace \in S^2, \lbrace v \rbrace \in S^2} \left( \sum_{\sigma \in X^{(3)}, \lbrace u,v \rbrace \subset \sigma} m(\sigma) \right)= \sum_{\lbrace u,v \rbrace \in  X^{(1)}, \lbrace u \rbrace \in S^2, \lbrace v \rbrace \in S^2} m(\tau) = m(S^2,S^2) .
\end{dmath*}  
Therefore
\begin{equation}
\label{k=2 thm - ineq1}
\dfrac{\varepsilon_2}{10} \Vert \phi \Vert  \geq m(K_2^2) + 3m(K_3^2) \geq m(K_2^2) + m(K_3^2).
\end{equation}
This yields that 
\begin{equation}
\label{k=2 thm - ineq2}
(1-\dfrac{\varepsilon_2}{10})  \Vert \phi \Vert \leq m(K_0^2 \cap supp(\phi)) + m(K_1^2 \cap supp (\phi) ).
\end{equation}
Denote 
$$\alpha = \dfrac{m(K_0^2 \cap supp(\phi))}{m(K_1^2 \cap supp (\phi) )} ,$$
(we'll also deal with the case where $m(K_1^2 \cap supp (\phi) )=0$).  Next, we'll prove two inequalities dependant on the values of $\alpha$, which much the two options stated in the theorem. \\
\textbf{Inequality 1} :
Notice that 
\begin{dmath*} 
\sum_{\lbrace v \rbrace \in S^2} \Vert \phi_{\lbrace v \rbrace} \Vert = m(K_1^2 \cap supp (\phi)) + 2 m(K_2^2 \cap supp (\phi))+ 3 m(K_3^2 \cap supp (\phi)) \geq m(K_1^2 \cap supp (\phi)).
\end{dmath*}
Therefore, by \eqref{k=2 thm - ineq2}, we have that 
$$(1-\dfrac{ \varepsilon_2}{10})  \Vert \phi \Vert \leq (1+ \alpha) m(K_1^2 \cap supp(\phi)) \leq  (1+ \alpha) \sum_{\lbrace v \rbrace \in S^2} \Vert \phi_{\lbrace v \rbrace} \Vert .$$
Using the fact that $\varepsilon_2 \leq 1$, this yields 
\begin{equation}
\label{k=2 thm - ineq3}
 \Vert \phi \Vert \leq \dfrac{10 (1+ \alpha)}{9} \sum_{\lbrace v \rbrace \in S^2} \Vert \phi_{\lbrace v \rbrace} \Vert .
\end{equation}
Next, we'll analyse the connection between $d \phi $ and $d_{\lbrace v \rbrace} \phi_{\lbrace v \rbrace}$ for $\lbrace v \rbrace \in S^2$. For $i=0,1,2,3,4$, denote
$$K_i^3 = \lbrace \sigma \in X^{(3)} : \vert \lbrace \lbrace v \rbrace \in S^2 : \lbrace v \rbrace \subset \sigma  \rbrace \vert = i \rbrace.$$
As before, we have that 
\begin{dmath*}
m(K_2^3) + 3m(K_3^3) + 6m(K_4^3) =\sum_{\eta \in K_2^3} m(\eta) +  \sum_{\eta \in K_3^2} 3 m(\eta)  +  \sum_{\eta \in K_4^2} 3 m(\eta)= \sum_{\sigma \in K_3^2} \left( \sum_{\lbrace u,v \rbrace \in  X^{(1)}, \lbrace u,v \rbrace \subset \eta, \lbrace u \rbrace \in S^2, \lbrace v \rbrace \in S^2} m(\eta) \right) +  \sum_{\sigma \in K_3^2} \left( \sum_{\lbrace u,v \rbrace \in  X^{(1)}, \lbrace u,v \rbrace \subset \eta, \lbrace u \rbrace \in S^2, \lbrace v \rbrace \in S^2} m(\eta) \right) + \sum_{\sigma \in K_4^2} \left( \sum_{\lbrace u,v \rbrace \in  X^{(1)}, \lbrace u,v \rbrace \subset \eta, \lbrace u \rbrace \in S^2, \lbrace v \rbrace \in S^2} m(\eta) \right)=    
\sum_{\lbrace u,v \rbrace \in  X^{(1)}, \lbrace u \rbrace \in S^2, \lbrace v \rbrace \in S^2} \left( \sum_{\sigma \in X^{(4)}, \lbrace u,v \rbrace \subset \sigma} m(\eta) \right)= \dfrac{1}{2} \sum_{\lbrace u,v \rbrace \in  X^{(1)}, \lbrace u \rbrace \in S^2, \lbrace v \rbrace \in S^2} m(\tau) = \dfrac{1}{2} m(S^2,S^2) .
\end{dmath*}
This inequality allows us to bound the contribution of simplices in $K_2^3,K_3^3,K_4^3$ to the sum of the norms of the localizations:
\begin{dmath*}
\sum_{\lbrace v \rbrace \in S^2} \Vert d_{\lbrace v \rbrace} \phi_{\lbrace v \rbrace} \Vert = \sum_{\lbrace v \rbrace \in S^2} \left( \sum_{\tau \in X_{\lbrace v \rbrace}, \tau \in supp (d_{\lbrace v \rbrace} \phi_{{\lbrace v \rbrace}})}  m_{{\lbrace v \rbrace}} (\tau) \right) \leq \\
\sum_{\lbrace v \rbrace \in S^2} \left( \sum_{\tau \in X_{\lbrace v \rbrace}, \tau \in supp (d_{\lbrace v \rbrace} \phi_{{\lbrace v \rbrace}}), \lbrace v \rbrace \cup \tau \in K_1^3}  m_{{\lbrace v \rbrace}} (\tau) \right) + 2 m(K_2^3) + 3 m(K_3^3) + 4 m(K_4^3) \leq \\
\sum_{\lbrace v \rbrace \in S^2} \left( \sum_{\tau \in X_{\lbrace v \rbrace}, \tau \in supp (d_{\lbrace v \rbrace} \phi_{{\lbrace v \rbrace}}), \lbrace v \rbrace \cup \tau \in K_1^3}  m_{{\lbrace v \rbrace}} (\tau) \right) +m(S^2, S^2) \leq \\
\sum_{\lbrace v \rbrace \in S^2} \left( \sum_{\tau \in X_{\lbrace v \rbrace}, \tau \in supp (d_{\lbrace v \rbrace} \phi_{{\lbrace v \rbrace}}), \lbrace v \rbrace \cup \tau \in K_1^3}  m_{{\lbrace v \rbrace}} (\tau) \right) + \dfrac{\varepsilon_2}{10} \Vert \phi \Vert = \\
\sum_{\lbrace v \rbrace \in S^2} \left( \sum_{\eta \in K_1^3, \lbrace v \rbrace \subset \eta, (\eta \setminus \lbrace v \rbrace) \in supp (d_{\lbrace v \rbrace} \phi_{{\lbrace v \rbrace}})}  m (\eta) \right) + \dfrac{\varepsilon_2}{10} \Vert \phi \Vert.
\end{dmath*}
Therefore, the main contribution to $\sum_{\lbrace v \rbrace \in S^2} \Vert d_{\lbrace v \rbrace} \phi_{\lbrace v \rbrace} \Vert$ comes from simplices in $K_1^3$:
\begin{equation}
\label{k=2 thm ineq4}
\sum_{\lbrace v \rbrace \in S^2} \Vert d_{\lbrace v \rbrace} \phi_{\lbrace v \rbrace} \Vert - \dfrac{\varepsilon_2}{10} \Vert \phi \Vert \leq \sum_{\lbrace v \rbrace \in S^2} \left( \sum_{\eta \in K_1^3, \lbrace v \rbrace \subset \eta, (\eta \setminus \lbrace v \rbrace) \in supp (d_{\lbrace v \rbrace} \phi_{{\lbrace v \rbrace}})}  m (\eta) \right)  .
\end{equation}

Next, observe that for every $\eta \in X^{(3)}$ and for every $\lbrace v \rbrace  \subset \eta$  we have that
$$d \phi (\eta ) = d_{ \lbrace v \rbrace} \phi_{\lbrace v \rbrace} (\eta \setminus \lbrace v \rbrace )  + \phi (\eta \setminus \lbrace v \rbrace),$$
where the addition above is in $\mathbb{F}_2$. Note that for $\eta \in K_1^3$ and $\lbrace v \rbrace  \subset \eta, \lbrace v \rbrace \in S^2$ we have that $\eta \setminus \lbrace v \rbrace \in K_0^2$. Therefore 
\begin{dmath*}
\Vert d \phi \Vert - \sum_{\lbrace v \rbrace \in S^2} \left( \sum_{\eta \in K_1^3, \lbrace v \rbrace \subset \eta, (\eta \setminus \lbrace v \rbrace) \in supp (d_{\lbrace v \rbrace} \phi_{{\lbrace v \rbrace}})}  m (\eta) \right) \geq \\
{ \sum_{\eta \in K_1^3 \cap supp (d \phi)} m(\eta) -  \sum_{\lbrace v \rbrace \in S^2} \left( \sum_{\eta \in K_1^3, \lbrace v \rbrace \subset \eta, (\eta \setminus \lbrace v \rbrace) \in supp (d_{\lbrace v \rbrace} \phi_{{\lbrace v \rbrace}})}  m (\eta) \right)  =} \\
 \sum_{\eta \in K_1^3 \cap supp (d \phi)} m(\eta) -  \sum_{\eta \in K_1^3} \left( \sum_{\lbrace v \rbrace \in S^2, \lbrace v \rbrace \subset \eta, \eta \setminus \lbrace v \rbrace \in supp (d_{\lbrace v \rbrace} \phi_{{\lbrace v \rbrace}})} m(\eta) \right) 
\geq  - \sum_{\eta \in K_1^3, d \phi (\eta ) =0,  d_{ \lbrace v \rbrace} \phi_{\lbrace v \rbrace} (\eta \setminus \lbrace v \rbrace )=1 \text{ for }  \lbrace v \rbrace \in S^2,  \lbrace v \rbrace \subset \eta} m(\eta) \geq - m(K_0^2 \cap supp (\phi) ) \geq - \alpha m(K_1^2 \cap supp (\phi)).
\end{dmath*}
Combined with \eqref{k=2 thm ineq4}, this yields 
\begin{equation}
\label{k=2 thm - ineq 5}
 \Vert d \phi \Vert \geq \sum_{\lbrace v \rbrace \in S^2} \Vert d_{\lbrace v \rbrace} \phi_{\lbrace v \rbrace} \Vert - \dfrac{\varepsilon_2}{10} \Vert \phi \Vert -\alpha m(K_1^2 \cap supp (\phi)) .
\end{equation}
Note that 
$$\sum_{\lbrace v \rbrace \in S^2} \Vert \phi_{\lbrace v \rbrace} \Vert \geq m(K_1^2 \cap supp (\phi)).$$
Combine this with \eqref{k=2 thm - ineq 5} and \eqref{k=2 thm - ineq3} and get 
$$ \Vert d \phi \Vert \geq \sum_{\lbrace v \rbrace \in S^2} \Vert d_{\lbrace v \rbrace} \phi_{\lbrace v \rbrace} \Vert - \dfrac{ \varepsilon_2 (1+\alpha)}{9} \sum_{\lbrace v \rbrace \in S^2} \Vert \phi_{\lbrace v \rbrace} \Vert -\alpha \sum_{\lbrace v \rbrace \in S^2} \Vert \phi_{\lbrace v \rbrace} \Vert .$$
Therefore 
$$\Vert d \phi \Vert \geq \sum_{\lbrace v \rbrace \in S^2} \Vert d_{\lbrace v \rbrace} \phi_{\lbrace v \rbrace} \Vert - (\dfrac{ \varepsilon_2 (1+\alpha)}{9} + \alpha) \sum_{\lbrace v \rbrace \in S^2} \Vert \phi_{\lbrace v \rbrace} \Vert.$$
\textbf{Inequality 2}:  For $i=0,1,2,3$ denote 
$$L_i = \lbrace \sigma \in X^{(2)} : \sigma \text{ has } i \text{ } \delta-\text{thick edges} \rbrace.$$
Observe that if for $\lbrace u,v,w \rbrace \in X^{(2)}$ the edge $\lbrace u,v \rbrace$ is $\delta$-thick, then the vertex $u$ will be $\delta$-thick in $X_{\lbrace v \rbrace}$ and the vertex $v$ will be $\delta$-thick in $X_{\lbrace u \rbrace}$. By this observation we get that every $\sigma \in L_3$ will have an edge with $2$ $\delta$-thick vertices in every link. By lemma \ref{small implies thin lemma}, this implies 
$$m(K_1^2 \cap L_3 \cap supp (\phi)) \leq \sum_{v \in R^2} 2 \varepsilon_1 \Vert \phi_{\lbrace v \rbrace} \Vert \leq 6 \varepsilon_1 \Vert \phi \Vert = \dfrac{\varepsilon_2}{10} \Vert \phi \Vert.$$
Similarly 
$$m(K_0^2 \cap (L_2 \cup L_3) \cap supp (\phi)) \leq \dfrac{\varepsilon_2}{10} \Vert \phi \Vert.$$
Next, we observe that
\begin{dmath*}
\sum_{\tau \in A_\delta} \Vert \phi_\tau \Vert = {3 m(L_0 \cap supp (\phi)) + 2 m(L_1 \cap supp (\phi)) + m(L_2 \cap supp (\phi))} \geq \\
 2 m(K_0^2 \cap (L_0 \cup L_1) \cap supp (\phi)) + m(K_1^2 \cap (L_0 \cup L_1 \cup L_2) \cap supp (\phi) ) \geq \\ \\
 \left( 2 m(K_0^2 \cap supp (\phi)) - 2 m(K_0^2 \cap (L_2 \cup L_3) \cap supp (\phi)) \right) +  \left( m(K_1^2 \cap supp (\phi)) - m(K_1^2 \cap L_3 \cap supp (\phi)) \right) \geq \\ \\
  2 m(K_0^2 \cap supp (\phi)) + m(K_1^2 \cap supp (\phi)) - \dfrac{3 \varepsilon_2}{10} \Vert \phi \Vert \geq (2\alpha + 1) m(K_1^2 \cap supp (\phi)) - \dfrac{3 \varepsilon_2}{10} \Vert \phi \Vert .
\end{dmath*}
Note that by \eqref{k=2 thm - ineq2}, we have that
$$(1-\dfrac{\varepsilon_2}{10})  \Vert \phi \Vert \leq (1+\alpha) m(K_1^2 \cap supp(\phi)).$$
Combined with the above inequality, this yields
$$\sum_{\tau \in A_\delta} \Vert \phi_\tau \Vert \geq  3 \left( \dfrac{1 + 2 \alpha}{3 + 3 \alpha} - \dfrac{2 \varepsilon_2}{15} \right) \Vert \phi \Vert.$$
Therefore $\phi$ is $( \frac{1 + 2 \alpha}{3 + 3 \alpha} - \frac{2 \varepsilon_2}{15}, \delta)$-thin. \\ \\
Therefore, we can conclude the two inequalities analysed above as follows: for every $\phi$ with $\alpha$ defined as above, we have that
\begin{enumerate}
\item For the set $S^2$ as above
$$\forall \lbrace v \rbrace \in S^2, \Vert \phi_{\lbrace v \rbrace} \Vert \geq C_1' m(X_{\lbrace v \rbrace}^{(1)}) ,$$
$$\Vert \phi \Vert \leq \dfrac{10 (1+ \alpha)}{9} \sum_{\lbrace v \rbrace \in S^2} \Vert \phi_{\lbrace v \rbrace} \Vert ,$$
and
$$\Vert d \phi \Vert \geq \sum_{\lbrace v \rbrace \in S^2} \Vert d_{\lbrace v \rbrace} \phi_{\lbrace v \rbrace} \Vert - (\dfrac{ \varepsilon_2 (1+\alpha)}{9} + \alpha) \sum_{\lbrace v \rbrace \in S^2} \Vert \phi_{\lbrace v \rbrace} \Vert.$$
\item $\phi$ is $( \frac{1 + 2 \alpha}{3 + 3 \alpha} - \frac{2 \varepsilon_2}{15}, \delta)$-thin.
\end{enumerate}
Therefore if $\alpha \leq \varepsilon_2$ (using the fact that $\varepsilon_2 \leq 1$), we have that 
$$\Vert \phi \Vert \leq \dfrac{20}{9} \sum_{\lbrace v \rbrace \in S^2} \Vert \phi_{\lbrace v \rbrace} \Vert .$$
We also have that
$$\Vert d \phi \Vert \geq \sum_{\lbrace v \rbrace \in S^2} \Vert d_{\lbrace v \rbrace} \phi_{\lbrace v \rbrace} \Vert -\dfrac{ 11 \varepsilon_2}{9} \sum_{\lbrace v \rbrace \in S^2} \Vert \phi_{\lbrace v \rbrace} \Vert.$$
and if $\alpha \geq \varepsilon_2$ (using the fact that $\varepsilon_2 \leq 1$), we have that $\phi$ is $(\frac{1}{3} + \frac{\varepsilon_2}{15}, \delta)$-thin. \\
Last, we note that in the case where $m(K_1^2 \cap supp (\phi)) = 0$, repeating the above calculation of the second inequality shows that 
$$\sum_{\tau \in A_\delta} \Vert \phi_\tau \Vert \geq  3 \left( \dfrac{2}{3} - \dfrac{2 \varepsilon_2}{15} \right) \Vert \phi \Vert,$$
and therefore in this case $\phi$ is $( \frac{2}{3} - \frac{2 \varepsilon_2}{15}, \delta)$-thin (note that $\varepsilon_2 \leq 1$ and therefore the theorem holds for this case as well). \\
\end{proof}

Next, we'll add a further assumption about the coboundary expansion of the links of vertices to deduce isoperimetric inequalities in the case $k=2$:

\begin{theorem}
\label{k=2 isoperimetric inequality}
For every $\epsilon >0$, there are constants  $\theta_2 = \theta_2 (\epsilon) < 1$, $C_2 = C_2 (\epsilon) <1$, $\varepsilon (\epsilon) >0$, such that for every simplicial complex $X$, if 
$$\min \lbrace \lambda_0, \lambda_1, \lambda_2 \rbrace \geq \theta_2,$$ 
and 
$$\forall \lbrace v \rbrace \in X^{(0)}, \epsilon_1 (X_{\lbrace v \rbrace}) \geq \epsilon,$$
where  $\epsilon_1 (X_{\lbrace v \rbrace})$ is the 1-coboundry expansion of $X_{\lbrace v \rbrace}$.
Then for every $0 \neq \phi \in C^2 (X,\mathbb{F}_2)$, we have that
$$\left( \phi \text{ is } \varepsilon \text{-locally minimal and } \Vert \phi \Vert \leq C_2 m(X^{(2)}) \right)  \Rightarrow \Vert d \phi \Vert \geq \dfrac{3 \epsilon}{10} \Vert \phi \Vert > 0.$$
\end{theorem}

\begin{proof}
Choose $\varepsilon_2 = \frac{35 \epsilon}{100}, \delta = \frac{\epsilon}{1000}$ and let $C_1' = C_1' (\varepsilon_2, \delta), C_2' = C_2' (\varepsilon_2, \delta),\theta_2' = \theta_2' (\varepsilon_2, \delta)$ be as in theorem \ref{k=2 alternative theorem} above. Choose
$$\varepsilon = \min \left\lbrace \dfrac{C_1'}{4},  \frac{\epsilon}{1000} \right\rbrace ,$$
$$\theta_2 = \max \left\lbrace \theta_2', \frac{2 + \frac{3 \epsilon}{10}}{ 2 + \frac{61 \epsilon}{2000}} \right\rbrace ,$$
$$C_2 = C_2' .$$
Fix $0 \neq \phi \in C^2 (X,\mathbb{F}_2)$ such that $\phi$ is $\varepsilon$-locally minimal and $\Vert \phi \Vert \leq C_2 m(X^{(2)})$. By theorem  \ref{k=2 alternative theorem} above at least one the the following occurs:
\begin{enumerate}
\item There is a set $S^2 \subset X^{(0)}$ such that 
$$\forall \lbrace v \rbrace \in S^2, \Vert \phi_{\lbrace v \rbrace} \Vert \geq C_1' m(X_{\lbrace v \rbrace}^{(1)}) ,$$
\begin{equation}
\label{k=2 isoper. proof - ineq1}
 \sum_{\lbrace v \rbrace \in S^2} \Vert \phi_{\lbrace v \rbrace} \Vert \geq \dfrac{9}{20} \Vert \phi \Vert,
\end{equation}
and
\begin{equation}
\label{k=2 isoper. proof - ineq2}
\Vert d \phi \Vert \geq \sum_{\lbrace v \rbrace \in S^2} \Vert d_{\lbrace v \rbrace} \phi_{\lbrace v \rbrace} \Vert -\dfrac{ 11 \varepsilon_2}{9} \sum_{\lbrace v \rbrace \in S^2} \Vert \phi_{\lbrace v \rbrace} \Vert.
\end{equation}
\item $\phi$ is $(\frac{1}{3} + \frac{\varepsilon_2}{15}, \delta)$-thin.
\end{enumerate}
We'll prove the needed inequality for each case. \\
\textbf{Case 1:} If there is a set $S^2$ as mentioned above, we have for each $\lbrace v \rbrace \in S^2$, we have that 
$$\Vert \phi_{\lbrace v \rbrace} \Vert \geq C_1' m(X_{\lbrace v \rbrace}^{(1)}) =  \dfrac{C_1'}{2} m(X_{\lbrace v \rbrace}^{(0)}) = \dfrac{C_1'}{2} m(\lbrace v \rbrace) .$$
By the fact that $\phi$ is $\varepsilon$-locally minimal and that $\varepsilon \leq \frac{C_1'}{4}$, we have for every $\lbrace v \rbrace \in S^2$ that 
\begin{dmath*}
{\min \lbrace \Vert \phi_{\lbrace v \rbrace}  - \varphi \Vert : \varphi \in B^1 (X_{\lbrace v \rbrace}, \mathbb{F}_2) \rbrace} \geq \\
\Vert \phi_{\lbrace v \rbrace} \Vert - \frac{C_1'}{4} m (\lbrace v \rbrace) \geq \Vert \phi_{\lbrace v \rbrace} \Vert  - \dfrac{1}{2} \Vert \phi_{\lbrace v \rbrace} \Vert  =\dfrac{1}{2} \Vert \phi_{\lbrace v \rbrace} \Vert.
\end{dmath*}
By our assumption on $\epsilon_1 (X_{\lbrace v \rbrace})$, we therefore have for each $\lbrace v \rbrace \in S^2$ that
$$\Vert d_{\lbrace v \rbrace} \phi_{\lbrace v \rbrace} \Vert \geq \epsilon \min \lbrace \Vert \phi_{\lbrace v \rbrace}  - \varphi \Vert : \varphi \in B^1 (X_{\lbrace v \rbrace}, \mathbb{F}_2)\rbrace \geq   \frac{\epsilon}{2} \Vert \phi_{\lbrace v \rbrace} \Vert .$$ 
Next, combine the above inequality \eqref{k=2 isoper. proof - ineq2} to get
\begin{dmath*}
\Vert d \phi \Vert \geq \sum_{\lbrace v \rbrace \in S^2} \Vert d_{\lbrace v \rbrace} \phi_{\lbrace v \rbrace} \Vert -\dfrac{ 11 \varepsilon_2}{9} \sum_{\lbrace v \rbrace \in S^2} \Vert \phi_{\lbrace v \rbrace} \Vert \geq \sum_{\lbrace v \rbrace \in S^2}   \Vert \phi_{\lbrace v \rbrace} \Vert  \left( \frac{\epsilon}{2} - \dfrac{ 11 \varepsilon_2}{9} \right).
\end{dmath*}
We choose $\varepsilon_2 = \frac{35\epsilon}{100}$ and therefore, this yields
\begin{dmath*}
\Vert d \phi \Vert \geq \dfrac{65 \epsilon}{900}\sum_{\lbrace v \rbrace \in S^2}  \Vert \Vert \phi_{\lbrace v \rbrace} \Vert.
\end{dmath*}
By \eqref{k=2 isoper. proof - ineq1}, we get that 
\begin{dmath*}
\Vert d \phi \Vert \geq \dfrac{65 \epsilon}{900} \dfrac{9}{20}\Vert \phi \Vert = \dfrac{65 \epsilon}{2000} \Vert \phi \Vert > \dfrac{3 \epsilon}{10} \Vert \phi \Vert,
\end{dmath*}
as needed.  \\
\textbf{Case 2:} If $\phi$ is $(\frac{1}{3} + \frac{\varepsilon_2}{15}, \delta)$-thin, then by lemma \ref{thin has laplacian inequality}, we have that 
$$\Vert d \phi \Vert  \geq   \left(  \left(   \dfrac{\frac{1}{3} + \frac{\varepsilon_2}{15}+1}{2}   - \delta  - \dfrac{\varepsilon}{2} \right)  3 \lambda_2 - 2 \right) \Vert \phi \Vert  .$$
Which can be simplified to
$$\Vert d \phi \Vert  \geq   \left(  \left( 2 + \dfrac{\varepsilon_2}{10}  - 3 \delta  - \dfrac{3 \varepsilon}{2} \right)  \lambda_2 - 2 \right) \Vert \phi \Vert.$$
We chose $\varepsilon_2 = \frac{35 \epsilon}{100}, \delta =\frac{ \epsilon}{1000}, \varepsilon \leq \frac{ \epsilon}{1000}$ and therefore this reads
$$\Vert d \phi \Vert  \geq   \left(  \left( 2 + \dfrac{35 \epsilon}{1000}  - \frac{3 \epsilon}{1000}  - \dfrac{3 \epsilon}{2000} \right)  \lambda_2 - 2 \right) \Vert \phi \Vert.$$
After more simplifying we get
$$\Vert d \phi \Vert  \geq   \left(  \left( 2 + \dfrac{61 \epsilon}{2000}  \right)  \lambda_2 - 2 \right) \Vert \phi \Vert.$$
We chose $\theta \geq \frac{2 + \frac{3 \epsilon}{10}}{ 2 + \frac{61 \epsilon}{2000}}$ and therefore we have that 
$$\Vert d \phi \Vert  \geq   \left(  2 + \dfrac{3 \epsilon}{10}  - 2 \right) \Vert \phi \Vert = \dfrac{3 \epsilon}{10} \Vert \phi \Vert ,$$
as needed.
\end{proof}

Combining the corollary \ref{k=0 isoperimetric inequality}, theorem \ref{k=1 isoperimetric inequality theorem}  with theorem \ref{decent in links theorem}, yields the following (which proves theorem \ref{isoperimetric inequalities theorem - introduction, k=1} stated in the introduction):
\begin{theorem}
For every $n >1$, there are constants $\frac{n-1}{n} < \Lambda_1 <1, \varepsilon >0, C_1 > 0$ such that for every weighted simplicial complex $X$ of dimension $n >1$ with $\lambda$-local spectral expansion, if $\lambda \geq \Lambda_1$, we have:
\begin{enumerate}
\item For every $0 \neq \phi \in C^0 (X,\mathbb{F}_2)$
$$\left( \phi \text{ is } \varepsilon \text{-locally minimal} \right) \Rightarrow \Vert d \phi \Vert >0.$$
\item For every $0 \neq \phi \in C^1 (X,\mathbb{F}_2)$
$$ \left( \phi \text{ is } \varepsilon \text{-locally minimal and } \Vert \phi \Vert \leq C_1 m(X^{(1)}) \right) \Rightarrow \Vert d \phi \Vert \geq \dfrac{1}{4} \Vert \phi \Vert >0.$$
\end{enumerate}

\end{theorem}

Combing the above with theorem \ref{k=2 isoperimetric inequality} (and adding further assumptions regarding the coboundary expansion of the links of vertices) yields the following:
\begin{theorem}
For every $n >2$, $\epsilon >0$, there are constants $\frac{n-1}{n} < \Lambda_2 <1, \varepsilon >0, C_1 > 0, C_2 >0$ such that for every weighted simplicial complex $X$ of dimension $n >2$ with $\lambda$-local spectral expansion, if $\lambda \geq \Lambda_2$ and if
$$\forall \lbrace v \rbrace \in X^{(0)}, \epsilon_1 (X_{\lbrace v \rbrace}) \geq \epsilon,$$
(where  $\epsilon_1 (X_{\lbrace v \rbrace})$ is the 1-coboundry expansion of $X_{\lbrace v \rbrace}$), we have that
\begin{enumerate}
\item For every $0 \neq \phi \in C^0 (X,\mathbb{F}_2)$
$$\left( \phi \text{ is } \varepsilon \text{-locally minimal} \right) \Rightarrow \Vert d \phi \Vert >0.$$
\item For every $0 \neq \phi \in C^1 (X,\mathbb{F}_2)$
$$ \left( \phi \text{ is } \varepsilon \text{-locally minimal and } \Vert \phi \Vert \leq C_1 m(X^{(1)}) \right) \Rightarrow \Vert d \phi \Vert \geq \dfrac{1}{4} \Vert \phi \Vert >0.$$
\item  For every $0 \neq \phi \in C^2 (X,\mathbb{F}_2)$
$$\left( \phi \text{ is } \varepsilon \text{-locally minimal and } \Vert \phi \Vert \leq C_2 m(X^{(2)}) \right) \Rightarrow \Vert d \phi \Vert \geq \dfrac{3 \epsilon}{10} \Vert \phi \Vert >0.$$
\end{enumerate}

\end{theorem}

Combining the above theorem with lemma \ref{topological overlap of the n-1 skeleton based on isoperimetric ineq}, gives the following theorem:
\begin{theorem}
\label{Topological overlap for 2-skeletons}
For $n>2$, $\epsilon>0$ there are constants $\frac{n-1}{n} < \Lambda_2 = \Lambda_2 (n)<1$, $1 \leq M, c = c (\Lambda_2, M, \epsilon)$, such that for any $n$-dimensional complex $X$, if:
\begin{enumerate}
\item $X$ has $\lambda$-local spectral expansion and $\lambda \geq  \Lambda_2$.
\item $$\forall \lbrace v \rbrace \in X^{(0)}, \epsilon_1 (X_{\lbrace v \rbrace}) \geq \epsilon.$$
\item $$\dfrac{\sup_{\sigma \in X^{(2)}} \vert \lbrace \eta \in X^{(n)} : \sigma \subset \eta \rbrace \vert }{\inf_{\sigma \in X^{(2)}} \vert \lbrace \eta \in X^{(n)} : \sigma \subset \eta \rbrace \vert } \leq M .$$
\end{enumerate}
Then the $2$-skeleton of $X$ has $c$-topological overlapping.
\end{theorem}

\section{Topological overlapping for $2$-skeletons of affine buildings}
The main difficulty of applying theorem \ref{Topological overlap for 2-skeletons} in order to construct a sequence of $2$-dimensional topological expanders is the second condition stated in the theorem, i.e., the condition bounding $\epsilon_1$ for each link of each vertex. Fortunately, such bounds exist for affine buildings that arise from BN-pairs. The subject of buildings is far to wide to present in the context of this article and the interested reader is referred to \cite{AbBrownBook} and references therein. Here we shall assume knowledge of the basic facts of BN-pairs and affine buildings. \\
The main result we'll use is the bound on the coboundary expansion of spherical buildings that arise from BN-pairs. This result was already mentioned in \cite{Grom}[page 457] and an explicit proof can be found in \cite{LMM}[Section 3]:
\begin{theorem}{\cite{LMM}[Corollary 3.6]}
Let $G$ be a group with a $BN$-pair and a spherical (i.e., finite) $(W,S)$ Coxeter group with $\vert S \vert = n+1$. Let $Y$ be the (finite) $n$-dimensional spherical building that arises for the $BN$-pair. Then for every $0 \leq k \leq n-1$, we have that 
$$\epsilon_k (Y) \geq \left({n+1 \choose k+2}^2 \vert W \vert \right)^{-1}.$$
\end{theorem}
We note that the bound on $\epsilon_1$ depends only on the type of the building and not on the thickness of the building. Next, let $\widetilde{X}$ be an affine $n$-dimensional building that arises from a group $G$ with an affine BN-pair constructed over a non-archimedean local field $F$ ,i.e., $F$ is a finite extension of either $\mathbb{Q}_p$ or $\mathbb{F}_q ((t))$. We recall that for every type of affine, we can choose $p$ or $q$ (related to $F$) to be large as we want. This choice determines the thickness of the building, but does not effect the type of the building ($W$ stays the same). Note that all the links of $\widetilde{X}$ (excluding $\widetilde{X}$ itself) are spherical and therefore for $X$ we have the desired bound on $\epsilon_1$ on the links of all vertices (as noted - this bound is not affected by the choice of the field $F$).  \\
Next, we recall that the spectral gaps of all $1$-dimensional spherical buildings were computed explicitly in \cite{FH} and for every type of spherical building with thickness $t$, the spectral gap is $ \geq 1- O(t^{-\frac{1}{2}})$. Therefore choosing $F = \mathbb{Q}_p$ or $\mathbb{F}_q ((t))$ with $p$ or $q$ large enough ensures that the spectral gap of all $1$ dimensional links of $X$ are greater that $\Lambda_2$ for theorem \ref{Topological overlap for 2-skeletons}. Therefore, by taking quotients of $\widetilde{X}$ that do not change the links of $\widetilde{X}$ we can get a topological expander:
\begin{corollary}
Let $\widetilde{X}$ be an $n$ dimensional ($n>2$) affine building that arises from group $G$ with an affine  $BN$-pair constructed over a non-archimedean local field $F = \mathbb{Q}_p$ or $\mathbb{F}_q ((t))$ with $p$ or $q$ large. Let $W$ be the Weyl group of the $BN$-pair. Let $\Gamma$ be a discrete subgroup of $G$ acting cocompactly on $X$ such that 
$$\forall \lbrace v \rbrace \in \widetilde{X}^{(0)}, \forall g \in \Gamma, d(g.v, v) >2,$$
where $d$ is the distance with respect to the metric of the $1$-skeleton on $\widetilde{X}$. Then there is $c=c(p \text{ or } q, W)>0$ such that $X=\widetilde{X} / \Gamma$ has $c$-topological overlapping. 
\end{corollary}

\begin{proof}
We apply theorem \ref{Topological overlap for 2-skeletons}. Note that the condition 
$$\forall \lbrace v \rbrace \in X^{(0)}, \forall g \in \Gamma, d(g.v, v) >2,$$
ensures that all $1,...,n-1$-dimensional links of $\widetilde{X}$ are isomorphic to links of $X$. Therefore choosing $p$ (or $q$) large enough ensures that $X / \Gamma$ has $\lambda$-local spectral gap, where $\lambda \geq \Lambda_2$. Also, it ensures that for every vertex $v$ of $X$ we have that
$$\epsilon_1 (X_{\lbrace v \rbrace}) \geq \left({n+1 \choose 3}^2 \vert W \vert \right)^{-1}.$$
Last, the constant $M$ in the theorem can be computed by the type of the building as a function of $p$ or $q$ (if $\widetilde{X}$ is of type $\widetilde{A}_n$, then $M=1$, but otherwise is it a function of $p$ or $q$).
\end{proof}

\appendix 
\section{Cheeger inequality for weighted graphs}
The aim of this appendix is to review the basic definition of a weighted graph and to state and prove the Cheeger inequality (and some of its consequences for weighted graphs). This appendix doesn't contain any new results and we provide the proofs merely for the sake of completeness. \\
For a graph $G=(V,E)$, a weight function is a function $m: V \cup E \rightarrow \mathbb{R}^+$ such that 
$$\forall v \in V, m (v) = \sum_{e \in E, v \in e} m(e) .$$ 
Denote the ordered edges of the graph as $\Sigma (1)$, i.e.,
$$\Sigma (1) = \lbrace (v,u) : v,u \in V, \lbrace u,v \rbrace \in E \rbrace.$$
Define 
$$C^0 (G,\mathbb{R} ) = \lbrace \phi : V \rightarrow \mathbb{R} \rbrace ,$$
$$C^1 (G,\mathbb{R} ) = \lbrace \phi : \Sigma (1) \rightarrow \mathbb{R} : \forall (u,v) \in \Sigma (1), \phi ((u,v)) = - \phi ((v,u)) \rbrace .$$
Define the differential $d: C^0 (G,\mathbb{R} ) \rightarrow C^1 (G,\mathbb{R} )$ as 
$$d \phi ((u,v)) = \phi (u) - \phi (v) .$$
Define inner products on $C^0 (G,\mathbb{R} ), C^1 (G,\mathbb{R} )$ as 
$$\forall \phi , \psi \in C^0 (G,\mathbb{R} ), \langle \phi , \psi \rangle = \sum_{v \in V} m (v) \phi (v) \psi (v) ,$$
$$\forall \phi , \psi \in C^1 (G,\mathbb{R} ), \langle \phi , \psi \rangle = \sum_{(u,v) \in \Sigma (1)} \frac{1}{2} m (\lbrace u,v \rbrace) \phi ((u,v)) \psi ((u,v)) .$$
Denote $\Vert . \Vert$ to be the norm with respect to the inner products defined above (this should not be confused with our use of $\Vert . \Vert$ in the body of this paper). Then $C^0 (G,\mathbb{R} ), C^1 (G,\mathbb{R} )$ are Hilbert spaces and we can define $d^*: C^1 (G,\mathbb{R} ) \rightarrow C^0 (G,\mathbb{R} )$ as the adjoint operator to $d$. The graph Laplacian is defined as
$$\Delta^+ : C^0 (G,\mathbb{R} ) \rightarrow C^0 (G,\mathbb{R} ), \Delta^+ = d^* d.$$
Note that by its definition, $\Delta^+$ is a positive operator.  
\begin{proposition}
Let $G=(V,E)$ be a graph with a weight function $m$. Then:
\begin{enumerate}
\item $d^* : C^1 (G, \mathbb{R} ) \rightarrow C^0 (G, \mathbb{R} )$ can be written explicitly as 
$$\forall \psi \in C^1 (G, \mathbb{R} ), \forall v \in V, d^* \phi (v) = \sum_{u \in V, (v,u) \in \sigma (1)} \dfrac{m (\lbrace u,v \rbrace)}{m(v)} \phi ((v,u)).$$
\item $\Delta^+$ can be written explicitly as
$$\forall \phi \in C^0 (G, \mathbb{R}),  \forall v \in V, \Delta^+ \phi (v) = \phi (v) - \dfrac{1}{m(v)} \sum_{u \in V, \lbrace u,v \rbrace \in E} m (\lbrace u,v \rbrace) \phi (u) .$$
\item If $G$ is connected, then $\Delta^+ \phi = 0$ if and only if $\phi$ is constant.
\end{enumerate}
 
\end{proposition}

\begin{proof}
\begin{enumerate}
\item Let $\phi \in C^0 (G, \mathbb{R})$ and $\psi \in C^1$, then 
\begin{dmath*}
\langle d \phi , \psi \rangle = \sum_{(u,v) \in \Sigma (1)} \frac{1}{2} m (\lbrace u,v \rbrace) d \phi ((u,v)) \psi ((u,v))  = \sum_{(u,v) \in \Sigma (1)} \frac{1}{2} m (\lbrace u,v \rbrace) (\phi (u) - \phi (v) ) \psi ((u,v)) = \sum_{(u,v) \in \Sigma (1)} \frac{1}{2} m (\lbrace u,v \rbrace) \left( \phi (u)  \psi ((u,v))  +  \phi (v)  \psi ((v,u)) \right) = \sum_{v \in V} m(v) \phi (v) \sum_{(v,u) \in \Sigma (1)} \dfrac{m (\lbrace u,v \rbrace)}{m(v)} \psi ((v,u)) = \langle \phi , d^* \psi \rangle.
\end{dmath*}
\item Let $\phi \in C^0 (G, \mathbb{R} )$ and $v \in V$, then
\begin{dmath*}
\Delta^+ \phi (v) = d^* d \phi (v) = \sum_{u \in V, (v,u) \in \sigma (1)} \dfrac{m (\lbrace u,v \rbrace)}{m(v)} d \phi ((v,u)) =\sum_{u \in V, (v,u) \in \sigma (1)} \dfrac{m (\lbrace u,v \rbrace)}{m(v)} \left( \phi (v) - \phi (u) \right) = \phi (v)- \sum_{u \in V, (v,u) \in \sigma (1)} \dfrac{m (\lbrace u,v \rbrace)}{m(v)} \phi (u) = \phi (v) - \dfrac{1}{m(v)} \sum_{u \in V, \lbrace u,v \rbrace \in E} m (\lbrace u,v \rbrace) \phi (u).
\end{dmath*}
\item Let $\mathbbm{1} \in C^0 (G, \mathbb{R})$, the constant function $\mathbbm{1} (v) = 1, \forall v \in V$. Then $d \mathbbm{1} \equiv 0$ and therefore $\Delta^+  \mathbbm{1} = 0$. On the other hand, if $\phi \in C^0 (G, \mathbb{R})$ is a function such that $\Delta^+ \phi = 0$, we can apply the maximum principle: let $v_0 \in V$ such that 
$$\forall v \in V,  \phi (v_0)  \geq  \phi (v)  .$$
Next,
$$0= \Delta^+ \phi (v_0) = \phi (v_0) - \dfrac{1}{m(v_0)} \sum_{u \in V, \lbrace u,v_0 \rbrace \in E} m (\lbrace u,v_0 \rbrace) \phi (u) ,$$
implies that 
$$\phi (v_0) =  \dfrac{1}{m(v_0)} \sum_{u \in V, \lbrace u,v_0 \rbrace \in E} m (\lbrace u,v_0 \rbrace) \phi (u).$$
The right hand side of the equation is an average $\phi (u)$'s such that for every $u$, $\phi (v_0)  \geq  \phi (u)$. Therefore the equality implies that 
$$\forall u \in V, \lbrace u,v_0 \rbrace \in E  \Rightarrow \phi (u) = \phi (v_0).$$
Iterating this argument (using the fact that $G$ is connected) yields that $\phi$ must be constant.
\end{enumerate}

\end{proof}

As a result of the above proposition, if $G$ is connected, then $\Delta^+$ has the eigenvalue $0$ with multiplicity $1$ and all the other eigenvalues are strictly positive. Denote by $\lambda (G)$ the smallest positive eigenvalue. Next, we can state the Cheeger inequality:
\begin{proposition}
Let $G = (V,E)$ be a connected graph. We introduce the following notations:
For $\emptyset \neq U \subseteq V$ denote 
$$m (U) = \sum_{v \in U} m (v) .$$
For $\emptyset \neq U_1, U_2 \subseteq V$ denote
$$m(U_1,U_2 ) = \sum_{u_1 \in U_1, u_2 \in U_2,  \lbrace u_1, u_2 \rbrace \in E} m (\lbrace u_1, u_2 \rbrace ).$$
Then:
\begin{enumerate}
\item (Cheeger inequality) For every $\emptyset \neq U \subsetneq V$, we have that
$$m (U, V \setminus U) \geq \lambda (G) \dfrac{m (U) m (V \setminus U)}{m(V)} .$$
\item For every $\emptyset \neq U \subsetneq V$, we have that
$$\dfrac{m(U)}{2}  \left(1 - \lambda (G) \frac{m(V \setminus U)}{m(V)} \right) \geq m (U,U)  .$$
\end{enumerate}
\end{proposition}

\begin{proof}
\begin{enumerate}
\item Let $\emptyset \neq U \subsetneq V$. Define $\phi \in C^{0} (G, \mathbb{R} )$ as
$$\phi (u) = \begin{cases}
\dfrac{1}{m (U)} & u \in U \\
 \dfrac{-1}{m (V \setminus U)} & u \in V \setminus U
\end{cases}.$$
Recall that $\mathbbm{1}$ is the constant function $1$ and note that
\begin{dmath*}
\langle \phi , \mathbbm{1} \rangle = \sum_{u \in V} m (u) \phi (u) 1 = \sum_{u \in U} m(u)  \dfrac{1}{w (U)} + \sum_{u \in V \setminus U} m(u) \dfrac{-1}{m (V \setminus U)} = 1 -1 =0.
\end{dmath*}
Therefore $\phi \perp \mathbbm{1}$. Recall that $\Delta^+$ is a positive operator (in particular, it has orthogonal eigenfunction), therefore
$$\langle \Delta^+ \phi , \phi \rangle \geq \lambda (G) \Vert \phi \Vert^2 .$$
Next, note the following
\begin{dmath*}
\langle \Delta^+ \phi , \phi \rangle = \langle d \phi , d \phi \rangle= \sum_{(u,v) \in \Sigma (1), u \in U, v \in V \setminus U} \frac{1}{2} m (\lbrace u,v \rbrace) \left( \dfrac{1}{m(U)} + \dfrac{1}{m(V \setminus U)} \right)^2 = m (U, V \setminus U)  \left( \dfrac{1}{m(U)} + \dfrac{1}{m(V \setminus U)} \right)^2. 
\end{dmath*}
Also,
\begin{dmath*}
\Vert \phi \Vert^2 = \sum_{u \in U} m(u) \dfrac{1}{m(U)^2} + \sum_{u \in V \setminus U} m(u) \dfrac{1}{m(V \setminus U)^2} = \dfrac{1}{m(U)} + \dfrac{1}{m(V \setminus U)} .
\end{dmath*}
Therefore, the inequality 
$$\langle \Delta^+ \phi , \phi \rangle \geq \lambda (G) \Vert \phi \Vert^2 ,$$
yields
$$ m (U, V \setminus U)  \left( \dfrac{1}{m(U)} + \dfrac{1}{m(V \setminus U)} \right)^2 \geq  \left( \lambda (G) \dfrac{1}{m(U)} + \dfrac{1}{m(V \setminus U)} \right),$$
which in turn yields
\begin{dmath*}
 m (U, V \setminus U) \geq \lambda (G)  \dfrac{1}{\dfrac{1}{m(U)} + \dfrac{1}{m(V \setminus U)}} = \lambda (G) \dfrac{m (U) m(V \setminus U )}{m (U) + m (V \setminus U)} = \lambda (G) \dfrac{m (U) m(V \setminus U )}{m(V)} ,
 \end{dmath*}
 and we are done.
 \item Let $\emptyset \neq U \subsetneq V$. Note that for every $u \in U$ we have that
 $$m (u) = \sum_{v \in V, \lbrace u, v \rbrace \in E} m (\lbrace u,v \rbrace) = \sum_{v \in U, \lbrace u, v \rbrace \in E} m (\lbrace u,v \rbrace) + \sum_{v \in V \setminus U, \lbrace u, v \rbrace \in E} m (\lbrace u,v \rbrace) .$$
 Summing on all $u \in U$ we get that
\begin{dmath*} 
 m(U) = \sum_{u \in U} \sum_{v \in U, \lbrace u, v \rbrace \in E} m (\lbrace u,v \rbrace) + \sum_{u \in U} \sum_{v \in V \setminus U, \lbrace u, v \rbrace \in E} m (\lbrace u,v \rbrace) = 2 m (U,U) + m(U, V \setminus U ).
 \end{dmath*}
 Using the Cheeger inequality proven above, we get that
$$m(U) \geq 2 m (U,U)  + \lambda (G) \frac{m(U) m(V \setminus U)}{m(V)} .$$
This yields that 
$$\dfrac{m(U)}{2}  \left(1 - \lambda (G) \frac{m(V \setminus U)}{m(V)} \right) \geq m (U,U)  .$$
\end{enumerate}
\end{proof}

\bibliographystyle{alpha}
\bibliography{bibl}
\end{document}